\documentclass[a4paper,11pt]{article}
\usepackage[latin1]{inputenc}
\usepackage[english]{babel}
\usepackage[T1]{fontenc}
\usepackage{amsfonts,amssymb,amsthm,amsmath,amsbsy}
\usepackage{mathrsfs}
\usepackage{bm}
\usepackage{enumerate}
\usepackage{fullpage}
\usepackage{indentfirst}
\usepackage[affil-it]{authblk}
\usepackage[all]{xy}
\usepackage{multicol}

\newtheorem{teor}{Theorem}
\newtheorem*{teor*}{Theorem}
\newtheorem{lemma}[teor]{Lemma}
\newtheorem{prop}[teor]{Proposition}
\newtheorem{corol}[teor]{Corollary}

\theoremstyle{definition}

\theoremstyle{remark}
\newtheorem{rmk}[teor]{Remark}

\newcommand{\Z}{\mathbb{Z}}
\newcommand{\Q}{\mathbb{Q}}
\newcommand{\R}{\mathbb{R}}
\newcommand{\N}{\mathbb{N}}

\newcommand{\ent}{\mathcal{O}}

\newcommand{\keywords}[1]{\noindent \textbf{Keywords:}\quad #1}
\newcommand{\ead}[1]{\noindent \textbf{Email adress:}\quad #1}
\newcommand{\msc}[1]{\textbf{2010 Mathematics Subject Classification:}\quad #1}
\newcommand\blfootnote[1]{\begingroup\renewcommand\thefootnote{}\footnote{#1}
	\addtocounter{footnote}{-1}\endgroup}

\title{Poincaré series of double coset representatives of Coxeter groups}
\author{Gianmarco Chinello}
\affil{Sapienza Università di Roma\\Dipartimento di Matematica G. Castelnuovo}
\date{}

\begin{document}
\maketitle
\begin{abstract}
	Let $(W,S)$ be a Coxeter system of finite rank and let $J,K\subset S$. 
	We study the rationality of the Poincaré series of the 
	set of representatives of minimal length of $(W_J,W_K)$-double cosets 
	of $W$: we conclude that it depends mostly on the rationality of the 
	Poincaré series of the normalizers of finite parabolic subgroups of $W$. 
	For affine Weyl groups, we prove that all these series are rational and we 
	give some explicit examples.
\end{abstract}
\blfootnote{\msc{20F55, 05E15}}
\blfootnote{\keywords{Coxeter groups, Poincaré series, Growth series, Double 
cosets.}}
\blfootnote{\ead{chinellogianmarco[at]gmail.com}}

\section*{Introduction}
The aim of this paper is to investigate the rationality of certain growth 
series (or Poincaré series) of some subsets of a Coxeter group. 

\medskip
Let $(W,S)$ be a Coxeter system of finite rank and let $\ell$ be its length 
function. 
For a subset $X$ of $W$, its \emph{Poincaré series} is 
\[X(t)=\sum_{x\in X}t^{\ell(x)}\in\Z[[t]].\]
It measures its growth relative to the generating set $S$.
For many subsets of $W$, this series is actually a rational function.
Actually, the well-know formula (see \cite[\S5.12]{Hum})
\begin{equation}\label{eq:Wt}
\sum_{J\subset S}(-1)^{|J|}\;\frac{W(t)}{W_J(t)}=
\left\{
\begin{array}{ll}
t^{\ell(w_0)}& \text{if } W \text{ is finite, }w_0 \text{ is the 
maximal length element in }W\\
0& \text{if } W \text{ is infinite}
\end{array}
\right.
\end{equation}
permits to prove that the 
Poincaré series of the group $W$ and of its parabolic 
subgroups $W_J=\langle J \rangle$ where $J\subset S$ are rational functions.
Moreover, if $W^J$ (resp. ${}^JW$) is the set of representatives of minimal 
length of left (resp. right) $W_J$-cosets of $W$, then 
\begin{equation}\label{eq:WJt}
W^J(t)={}^JW(t)=\frac{W(t)}{W_J(t)}
\end{equation}
and so it is a rational function.
In  \cite{Vis} we can find other examples of rational Poincaré series.

\medskip
The purpose of the present paper is to obtain generalizations of (\ref{eq:Wt})
and to study the rationality of the Poincaré series of the 
set ${}^JW^K={}^JW\cap W^K$ of representatives of minimal length of 
$(W_J,W_K)$-double cosets 
of $W$ where $J,K\subset S$.
The main difference from the case of a single coset, in which (\ref{eq:WJt}) 
holds, is that in our case there is not such a formula. 
So, even if we know that $W_J(t)$ and $W_K(t)$ are rational we can not conclude 
that ${}^JW^K(t)$ is rational.

\medskip
The idea is to consider the sets
\[p^{S'}_{Q,J,K}=\{x\in{}^J W_{S'}^K\,|\,K\cap 
x^{-1}Jx=Q\}\]
for every $Q,J,K\subset S'\subset S$. They form a partition of  ${}^JW_{S'}^K$ 
and so the study of rationality of ${}^JW_{S'}^K(t)$ is reduced to that of 
$p^{S'}_{Q,J,K}(t)$. 
Investigating the relations between these series, we simplify the problem 
obtaining the following theorem.
\begin{teor}
	The series $p^{S'}_{Q,J,K}(t)$ is a rational function for every 
	$Q,J,K\subset S'\subset S$ if and only if 
	$\{x\in	W_{S'}\,|\, xLx^{-1}=M\}(t)$ is 
	a rational function for every $L,M\subset S'$ with $W_L$ finite. 
\end{teor}
\noindent
This result does not prove, in general, the rationality of 
${}^JW^K(t)$, but simplifies the question considerably. 
In particular, the problem is essentially reduced to the study of the 
rationality of the Poincaré series of the normalizers $N_W(W_L)$ of finite 
parabolic subgroups
$W_L$ of $W$ since $\{x\in	W\,|\, xLx^{-1}=L\}(t)=N_W(W_L)(t)/W_L(t)$.
A way to prove that $N_W(W_L)(t)$ is rational, would be to use the results in 
\cite{D82,BH} about the normalizers of parabolic subgroups.

\medskip
In the last part of this paper, we investigate the case of affine Weyl groups.
Let $(\mathcal{W},\mathcal{S})$ be a finite Coxeter system and let 
$\widetilde{\mathcal{W}}$ be the associated affine Weyl group. 
Using the fact that $\widetilde{\mathcal{W}}$ is the semidirect product of 
$\mathcal{W}$ with a normal subgroup, we 
prove the following theorem.
\begin{teor}
	The series $p_{Q,J,K}(t)$, ${}^J \widetilde{\mathcal{W}}^K(t)$ and 
	$N_{\widetilde{\mathcal{W}}}(\widetilde{\mathcal{W}}_J)(t)$ are rational 
	functions for every $Q,J,K\subset \mathcal{S}$.
\end{teor}
\noindent
Furthermore, we give an 
explicit method to calculate these rational functions and we give some examples.

\medskip
An important consequence of Theorem 2 will appear in a later paper 
\cite{CCW}. 
We summarize it briefly: let $F$ be a local field with residue field of 
cardinality $q=p^f$, let $G$ be the group of $F$-points of a split semisimple 
simply connected algebraic group over $F$ and let $\widetilde{\mathcal{W}}$ be 
the associated affine Weyl group. 
We can associate to some open compact subgroup $\ent$ of $G$, the 
$\zeta$-function
$$\zeta_{G,\ent}(s)=\sum_{x\in \ent\backslash G/\ent}[\ent:\ent\cap x\ent 
x^{-1}]^{-s}$$ where $s\in\mathbb{C}$.
In \cite{CCW}, we found an explicit expression of this series as a 
function of $p_{Q,J,K}(q^{-s})$ and, using Theorem 2, it turns 
out to be meromorphic (and holomorphic if $\ent$ is a pro-$p$-group).
Furthermore, using this expression, we prove that $\zeta_{G,\ent}(-1)$ is 
closely related with the Euler-Poincaré measure (or characteristic) of $G$ (see 
\cite[\S3]{Serre} for the definition). 
Other evidences of the connection between values of $\zeta$-functions and 
Euler-Poincaré characteristics can be found in \cite{Serre, Harder, Brown00} 
and in \cite[Ch.IX \S8]{Brown82}.

\medskip
The paper is organized as follows: in section \ref{sec:Coxeter} we introduce 
some preliminaries about Coxeter systems. 
Section \ref{section:mainthm} is devoted to
the proof of the main theorem (Theorem \ref{teor:equivalenze}) in the case of 
general Coxeter groups of finite rank.
In section \ref{section:coni} we recall some results about rational cones in 
order to use them in section \ref{section:affine} where we investigate the case 
of affine Weyl groups, and we prove that $p_{Q,J,K}(t)$, ${}^J 
\widetilde{\mathcal{W}}^K(t)$ and 
$N_{\widetilde{\mathcal{W}}}(\widetilde{\mathcal{W}}_J)(t)$ are rational 
functions whenever $Q,J, K\subset 
\mathcal{S}$ (Theorem \ref{teor:riassunto}).
Finally, in section \ref{sec:examples}, we give some explicit examples of the 
computation of these rational functions.

\numberwithin{teor}{section}
\numberwithin{equation}{section}

\section{Coxeter systems}\label{sec:Coxeter}
We recall that a \emph{Coxeter group} $W$ is a group with a presentation
\[W=\langle s\in S\,|\, (st)^{m(s,t)}=1 \text{ for every }s,t\in 
S\rangle\]
where $m(s,t)\in \Z_{>0}\cup\{\infty\}$, $m(s,t)=m(t,s)$ and $m(s,t)=1$ if and 
only if $s=t$. We call \emph{Coxeter system} the pair $(W,S)$ and \emph{rank} 
of $W$ the cardinality of $S$. 

\medskip
Throughout this section, let $(W,S)$ be a Coxeter system of finite rank.
The group $W$ can be represented faithfully as a reflection group on a real 
vector space $V$ with basis $\Sigma$ in bijective correspondence with $S$. We 
write $s_\alpha\in S$ the generator of $W$ corresponding to $\alpha\in\Sigma$, 
while if $J\subset S$ we write $\hat{J}$ for the corresponding subset of 
$\Sigma$.
The set $\bm\Phi=W\Sigma$ is called \emph{root system} of $W$ and we have 
$\bm\Phi=\bm\Phi^+\sqcup \bm\Phi^-$ where the elements of $\bm\Phi^+$ are 
nonnegative linear combinations of elements of $\Sigma$ and 
$\bm\Phi^-=-\bm\Phi^+$.

\medskip
Let $\ell:W\longrightarrow \N$ be the length function of $W$ with respect to 
$S$. For every $w\in W$ we have $\ell(w)=|N(w)|$ where $N(w)=\bm\Phi^+\cap 
w^{-1}\bm\Phi^-=\{\alpha\in\bm\Phi^+\,|\,w\alpha\in\bm\Phi^-\}$. In particular, 
we have 
\[
\ell(ws_\alpha)=\left\{
\begin{array}{ll}
\ell(w)-1 & \text{if } \alpha\in N(w)\\
\ell(w)+1 & \text{if } \alpha\notin N(w)
\end{array}
\right.
\]
for every $w\in W$ and $\alpha\in\Sigma$. The following is a well known result.

\begin{lemma}\label{lemma:rightdivisor}
	Let $u,w\in W$. Then $\ell(wu)=\ell(w)-\ell(u)$ if and only if 
	$N(u^{-1})\subset N(w)$.
\end{lemma}

For every $w\in W$, let
\[A_L(w):=\{s\in S\,|\,\ell(sw)>\ell(w)\} \quad \text{and}\quad A_R(w):=\{s\in 
S\,|\,\ell(ws)>\ell(w)\}\]
be the left and right \emph{ascent set} of $w$ and let $D_L(w):=S\setminus 
A_L(w)$ and 
$D_R(w):=S\setminus A_R(w)$ be the left and right \emph{descent set} of $w$.

\medskip 
For every $J\subset S$, the subgroup $W_J$ of $W$ generated by $J$ is called 
\emph{(standard) parabolic subgroup} of $W$. Then $(W_J,J)$ is a Coxeter system 
and we denote $\bm\Phi_J=W_J\hat{J}$ the 
corresponding root system. 
We consider the sets
\begin{align*}
	{}^JW&:=\{w\in W\,|\,J\subset A_L(w)\}=\{w\in W\,|\,\ell(sw)>\ell(w) 
	\,\forall\, 
	s\in J\},\\
	W^J&:=\{w\in W\,|\,J\subset A_R(w)\}=\{w\in W\,|\,\ell(ws)>\ell(w) 
	\,\forall\, 
	s\in J\}.
\end{align*}
We remark that $w\in {}^JW$ if and only if $w^{-1}\in W^J$.
According to \cite[\S2.3.2]{AB}, the elements of ${}^JW$ (resp. $W^J$) 
are the representatives of minimal length of right (resp. left) $W_J$-cosets of 
$W$. 
Furthermore, every element $w\in W$ can be written uniquely as 
$w=w_1w_2=w_3w_4$ with $w_1,w_4\in W_J$, $w_2\in {}^JW$ and $w_3\in W^J$ and we 
have
\begin{equation}\label{eq:lunghezze}
	\ell(w)=\ell(w_1)+\ell(w_2)=\ell(w_3)+\ell(w_4).
\end{equation}

A subset $K\subset S$ is said to be \emph{spherical} if $W_K$ is finite (or 
equivalently if $\bm\Phi_K$ is finite) and it is said to be \emph{irreducible} 
if $W_K$ cannot be nontrivially expressed as a direct product of smaller 
standard parabolic subgroups.
An arbitrary parabolic subgroup $W_J$ of $W$ can be written as a direct product 
of parabolic subgroups corresponding to irreducible subsets of $J$; we call 
these subsets the \emph{components} of $J$.

\medskip
For every $J,K\subset S$ we denote $J\perp K$ if $J\cap K=\emptyset$ and 
$jk=kj$ for every $j\in J$ and $k\in K$. We also denote $J^\perp=\{s\in 
S\setminus J\,|\,sj=js\;\forall j\in J\}$.
Clearly, if $J_1$ and $J_2$ are two components of $J$, then $J_1\perp J_2$.

\medskip
For every $X,Y\subset W$, we denote by  
$C_X(Y)=\{x\in X|\,xy=yx\;\forall\,y\in Y \}$ 
the centralizer of $Y$ in $X$ and by $N_X(Y)=\{x\in X|\, xY=Yx\}$ 
the normalizer of $Y$ in $X$.

\begin{rmk}\label{rmk:wHJ}
If $K\subset S$ is spherical, there exists a unique element $w_K$ of maximal 
length in $W_K$ and $N(w_K)=\bm\Phi_K^+$. 
More generally, if we take $J\subset H\subset S$ such that $W_H^{J}$ is finite 
(or equivalently $\bm\Phi_H\setminus \bm\Phi_J$ is finite), 
then there exists a unique element $w(H,J)\in W_H^J$ of maximal 
length: if $H_1$ is the union of components of $H$ 
which intersect $H\setminus J$ and $J_1=J\cap H_1$ then $H_1$ is spherical and 
$w(H,J)=w(H_1,J_1)=w_{H_1}w_{J_1}$ (see \cite{BH}).
In particular, we have $N(w(H,J))=\bm\Phi_H^+\setminus \bm\Phi_J^+$, 
$J'=w(H,J)Jw(H,J)^{-1}$ is a subset of $H$, 
$w(H,J)^{-1}=w(H,J')$,
$D_R(w(H,J))=H\setminus J$ and 
$D_L(w(H,J))=H\setminus J'$.
\end{rmk}

\subsection{Double cosets}
Throughout this paragraph, let $J,K\subset S$. We state some properties about 
the set of representatives of minimal length 
of $(W_J,W_K)$-double cosets of $W$.

\medskip
The elements of ${}^JW^K:={}^JW\cap W^K$ form a set of representatives 
of $(W_J,W_K)$-double cosets of $W$ and, by \cite[Proposition 2.23]{AB}, every 
$x\in 
{}^JW^K$ is the unique element of 
minimal length in $W_JxW_K$. 
Furthermore, if $x\in {}^JW^K$ then by \cite[Lemma 2.25]{AB} we have
\begin{equation}\label{eq:lemma2.25}
	W_K\cap x^{-1} W_J x=W_{K\cap x^{-1} J x} .
\end{equation}

\begin{lemma}\label{lemma:compatibilita}\mbox{}
	\begin{enumerate}[(a)]
		\item If $w\in {}^JW$ and $s\in S$ then there are three possibilities:
		\begin{enumerate}[(i)]
			\item $\ell(ws)=\ell(w)+1$ and $ws\in{}^JW$;
			\item $\ell(ws)=\ell(w)+1$ and $ws=s'w$ with $s'\in J$;
			\item $\ell(ws)=\ell(w)-1$. In this case we have $ws\in{}^J W$.
		\end{enumerate}
		\item If $x\in {}^JW$, $v\in W_J$ and $H\subset J$ then $vx\in {}^HW$ 
		if and only if $v\in{}^HW$.
		\item If $x\in {}^JW^K$ and $u\in W_K$ then $xu\in {}^JW$ if and only 
		if $u\in{}^{K\cap x^{-1}Jx}W$.
	\end{enumerate}
\end{lemma}

\begin{proof}
	The point $(a)$ is a consequence of properties (E) and (F) at page 79 of 
	\cite{AB}, while $(b)$ follows by (\ref{eq:lunghezze}) since 
	$\ell(svx)-\ell(vx)=\ell(sv)-\ell(v)$ for every $s\in H$.
	It remains to prove $(c)$. We denote $Q=K\cap x^{-1}Jx$. If there exists 
	$s=x^{-1}s'x\in Q$ 	such that $\ell(su)<\ell(u)$ then 
	$\ell(s'xu)=\ell(xsu)=\ell(x)+\ell(su)<\ell(xu)$.
	On the other hand, we prove by induction on $\ell(u)$ that 
	$u\in{}^{Q}W$ implies $xu\in {}^JW$. If $\ell(u)=1$ then $u$ 
	belong to $K$ but not to $x^{-1}Jx$ and so by $(a)$ we have $xu\in {}^JW$.
	Now, let $u\in{}^{Q}W$ with $\ell(u)>1$ and let $u'=us\in W_K$ with 
	$\ell(u')<\ell(u)$. 
	Then by $(a)$ we have $u'\in {}^{Q}W$ and so by induction 
	hypothesis $xu'\in {}^JW$. Now, if we suppose by contradiction that 
	$xu's\notin {}^JW$, then by $(a)$ there exists 
	$s'\in J$ such that $xu's=s'xu'$. We obtain $y=x^{-1}s'x\in W_K\cap 
	x^{-1}W_Jx=W_Q$ and $\ell(yu)=\ell(u')<\ell(u)$ which contradicts $u'\in 
	{}^{Q}W$.
\end{proof}

\begin{lemma}\label{lemma:decwtilde}
	Every element $w$ of $W$ can be written in a unique way as $w=yxz$ with 
	$y\in W_J$, $x\in{}^JW^K$ and  $z\in {}^{K\cap x^{-1}Jx}W_K$. Moreover, we 
	have $\ell(w)=\ell(y)+\ell(x)+\ell(z)$.
\end{lemma}

\begin{proof}
	Let $w\in W$. Since  
	\[W=\bigsqcup_{x\in{}^JW^K}W_JxW_K,\]
	there exists a unique $x\in{}^JW^K$ such that $w\in W_JxW_K$. If we denote 
	$Q=K\cap x^{-1}Jx$, we have
	$W_JxW_K=W_JxW_Q{}^QW_K=W_Jx{}^QW_K$.	
	Now, if $w=y_1xz_1=y_2xz_2$ with $y_1,y_2\in W_J$ and 
	$z_1,z_2\in {}^{Q}W_K$, then the element
	$z_2z_1^{-1}=x^{-1}y_2^{-1}y_1x$ belongs to $W_K\cap x^{-1} W_J x=W_{Q}$ 
	and so $z_1=z_2$. It also follows that $y_1=y_2$.
	Now, by Lemma \ref{lemma:compatibilita}$(c)$ we have $xz_1\in {}^JW$ and 
	then
	$\ell(w)=\ell(y_1xz_1)=\ell(y_1)+\ell(xz_1)=\ell(y_1)+\ell(x)+\ell(z_1)$.
\end{proof}

One has to be careful, because $\ell(yxz)=\ell(y)+\ell(x)+\ell(z)$ 
does not hold for 
every $y\in W_J$, $x\in {}^JW^K$ and $z\in W_K$ (see Lemma 
\ref{lemma:compatibilita}$(c)$).

\section{The main theorem}\label{section:mainthm}
Throughout this section, let $(W,S)$ be a Coxeter system of finite rank.
We introduce some subsets of 
$W$ that form a partition of 
${}^J W^K$ and whose Poincaré series will allow us to investigate the 
rationality of ${}^J W^K(t)$.

\medskip 
For every $Q,J,K\subset S'\subset S$ we define the set
\begin{equation}\label{eq:defpQJK}
	p^{S'}_{Q,J,K}=\{x\in{}^J W_{S'}^K\,|\,K\cap 
	x^{-1}Jx=Q\}=\{x\in{}^J W_{S'}^K\,|\,\hat{K}\cap 
	x^{-1}\hat{J}=\hat{Q}\}
\end{equation}
and we write $p^{S}_{Q,J,K}=p_{Q,J,K}$. We are interested in the Poincaré 
series $p^{S'}_{Q,J,K}(t)$.

\medskip
We list some properties of these objects.
\begin{itemize}
	\item if $Q\not\subset K$ then $p^{S'}_{Q,J,K}=\emptyset$;
	\item if $|J|<|Q|$ then $p^{S'}_{Q,J,K}=\emptyset$;
	\item if $K=S'$ then $p^{S'}_{Q,J,S'}$ is equal to $\{1\}$ if 
	$Q=J$ and $\emptyset$ otherwise;
	\item if $J=S'$ then $p^{S'}_{Q,S',K}$ is equal to $\{1\}$ if 
	$Q=K$ and $\emptyset$ otherwise;
	\item if $K=\emptyset$ then $p^{S'}_{Q,J,\emptyset}$ is equal to 
	${}^JW_{S'}$ if $Q=\emptyset$ and $\emptyset$ otherwise;
	\item if $J=\emptyset$ then $p^{S'}_{Q,\emptyset,K}$ is equal to 
	$W^K_{S'}$ if $Q=\emptyset$ and $\emptyset$ otherwise;
	\item if $Q=J=K$ then 
	$p^{S'}_{K,K,K}=
	\{x\in{}^KW_{S'}^K\,|\,x^{-1}W_Kx=W_K\}=
	W_{S'}^K\cap 	N_{W_{S'}}(W_K)$
	and so \[\displaystyle{p^{S'}_{K,K,K}(t)= 
	\frac{N_{W_{S'}}(W_K)(t)}{W_K(t)}};\]
	\item by definition we have 
	\begin{equation}\label{eq:JWK}
		{}^J W_{S'}^K=\bigsqcup_{Q\subset K}p^{S'}_{Q,J,K};
	\end{equation} 
	\item by Lemma 
	\ref{lemma:decwtilde} we have
	\begin{equation}\label{eq:Wtilde}
		W_{S'}(t)=\sum_{x\in {}^JW_{S'}^K} 
		\frac{W_J(t)W_K(t)}{W_{K\cap x^{-1}Jx}(t)}\;t^{\ell(x)}
		=\sum_{Q\subset K}\frac{W_J(t)W_K(t)}{W_{Q}(t)}\;p^{S'}_{Q,J,K}(t).
	\end{equation}	
\end{itemize}

In order to study the rationality of $p^{S'}_{Q,J,K}(t)$, in the following 
paragraphs we will find some relations between these series (see Proposition 
\ref{prop:razionale2} and Corollary \ref{coroll:riduzione}). 

\subsection{First relation}

\begin{lemma}\label{lemma:razionale1}
	Let $J,K,K'\subset S'\subset S$ with $K\subset K'$.
	Then every element $x\in {}^JW_{S'}^K$ can be written in a unique way as 
	$x=yz$ with $y\in{}^JW_{S'}^{K'}$ and $z\in {}^{K'\cap y^{-1}Jy}W_{K'}^K$. 
\end{lemma}

\begin{proof}
	Let $x\in {}^JW_{S'}^K$. By Lemma \ref{lemma:decwtilde} we can write $x=yz$ 
	with $y\in {}^JW_{S'}^{K'}$ and $z\in {}^{K'\cap y^{-1}Jy}W_{K'}$. 
	By Lemma \ref{lemma:compatibilita}$(b)$, since $y\in W_{S'}^{K'}$ and 
	$yz=x\in W_{S'}^K$ we have $z\in W_{S'}^K$.
\end{proof}

\begin{prop}\label{prop:razionale2}
	Let $Q,J,K,K'\subset S'\subset S$ with $K\subset K'$.
	Then
	\begin{equation}\label{eq:relazionep}
		p^{S'}_{Q,J,K}(t)=\sum_{Q'\subset K'} 
		p^{K'}_{Q,Q',K}(t)\;p^{S'}_{Q',J,K'}(t).
	\end{equation}
\end{prop}

\begin{proof}
	We can take $Q\subset K$ and $|Q|\leq |J|$.
	In order to prove the proposition, we have to show
	\begin{equation}\label{eq:proof1.5}
		p^{S'}_{Q,J,K}=
		\bigsqcup_{Q'\subset K'}
		\bigsqcup_{\substack{y\in{}^JW_{S'}^{K'}\\ K'\cap y^{-1}Jy=Q'}}
		\bigsqcup_{\substack{z\in {}^{Q'}W_{K'}^K\\ K\cap z^{-1}Q'z=Q}}
		\{yz\}.
	\end{equation}
	Let $x\in p^{S'}_{Q,J,K}$. By Lemma 
	\ref{lemma:razionale1} we can write in a unique way $x=yz$ with 
	$y\in{}^JW_{S'}^{K'}$ and $z\in {}^{K'\cap y^{-1}Jy}W_{K'}^K$  
	and we have $Q=K\cap z^{-1}y^{-1}Jyz$.
	We have to prove $K\cap z^{-1}Q'z=Q$ where $Q'=K'\cap y^{-1}Jy$.
	Since $yz=x\in {}^JW_{S'}^K$, by (\ref{eq:lemma2.25}) we have 
	\[
	W_{K\cap z^{-1}y^{-1}Jyz}=W_K\cap z^{-1}y^{-1}W_Jyz
	=W_K\cap W_{K'}\cap z^{-1}y^{-1}W_Jyz\\
	=W_K\cap z^{-1}(W_{K'}\cap y^{-1}W_Jy)z
	\]
	Since $y\in {}^JW_{S'}^{K'}$, it is equal to
	$W_K\cap z^{-1}W_{K'\cap y^{-1}Jy}z=W_K\cap z^{-1}W_{Q'}z$ and since $z\in {}^{Q'}W_{K'}^{K}$ it is equal to $W_{K\cap z^{-1}Q'z}$.
	We obtain $Q=K\cap z^{-1}Q'z$ and so the first inclusion in 
	(\ref{eq:proof1.5}).\\
	Now, if we take $Q'\subset K'$, $y\in p^{S'}_{Q',J,K'}$ and $z\in 
	p^{K'}_{Q,Q',K}$ then we have $x=yz\in {}^JW_{S'}^K$ by Lemma 
	\ref{lemma:compatibilita}.
	Furthermore, we have just proved above that $W_{K\cap 
	z^{-1}y^{-1}Jyz}=W_{K\cap 
	z^{-1}Q'z}$ and so $K\cap x^{-1}Jx=Q$ which proves (\ref{eq:proof1.5}) and 
	so the lemma.
\end{proof}

In the sum of (\ref{eq:relazionep}) we can take $Q'$ such that $|Q|\leq 
|Q'|\leq |J|$, otherwise the summand is $0$.
We observe that if we take $Q=K=\emptyset$ in (\ref{eq:relazionep}) we 
obtain (\ref{eq:Wtilde}).

\begin{rmk}\label{rmk:M}
For every $K\subset S'\subset S$, consider the matrix 
\[M_{K,S'}:=\Big(p^{S'}_{Q,J,K}(t)\Big)_{Q\subset K,J\subset S'}\in 
\mathrm{Mat}_{2^{|K|}\times 2^{|S'|}}(\Z[[t]]).\]
Proposition \ref{prop:razionale2} says that $M_{K,S'}=M_{K,K'}M_{K',S'}$ for 
every 
$K\subset K'\subset S'\subset S$.
This allows us to reduce the problem of studying the rationality of 
$p^{S'}_{Q,J,K}(t)$ to the case where $|K|=|S'|-1$.
\end{rmk}

\subsection{Second relation}

\begin{lemma}\label{lemma:AB}
	Let $A\subset \Sigma$ and $w\in W$ such that $wA\subset \Sigma$. Then for 
	every 
	$B\subset N(w)\cap\Sigma$ we have 
	$\bm\Phi^+_{A\cup B}\setminus\bm\Phi^+_{A}\subset N(w)$.
\end{lemma}

\begin{proof}
	Let $\alpha=\sum_{a\in A}\lambda_a a+\sum_{b\in B}\lambda_b b\in 
	\bm\Phi^+_{A\cup B}\setminus\bm\Phi^+_{A}$ and $w\alpha=\sum_{c\in 
	\Sigma}\mu_c c$. Then there exists $d\in \Sigma\setminus wA$ such that 
	$\mu_d\neq 
	0$, otherwise $\alpha=\sum_{c\in wA}\mu_c (w^{-1}c)\in\bm\Phi_A$. 
	Since $wb\in\bm\Phi^-$ for every $b\in B$, the coefficient $\mu_d$ must be 
	a negative linear combination of $\lambda_b$ and so $w\alpha\in\bm\Phi^-$.
\end{proof}
	
\begin{lemma}\label{lemma:uguaglianzainsiemi}
	Let $J,K\subset S$ and $Q\subset K$ with $W_K^Q$ finite. Let $w(K,Q)$ be 
	the element of maximal length in $W_K^Q$ and let $Q'=w(K,Q)Qw(K,Q)^{-1}$. 
	Then we 
	have
	\[\{xw(K,Q')\,|\,x\in p_{Q,J,K}\}=\{y\in p_{Q',J,Q'}\,|\, K\setminus 
	Q'\subset 
	D_R(y)\}.\]
\end{lemma}

\begin{proof}
	Let $v=w(K,Q)$. By remark \ref{rmk:wHJ} we have $v^{-1}=w(K,Q')$, 
	$D_R(v)=K\setminus Q$ and 
	$D_L(v)=K\setminus Q'$.
	If we take $x\in p_{Q,J,K}$  then $y=xv^{-1}$ is in ${}^JW$ by Lemma 
	\ref{lemma:compatibilita}$(c)$ and we have 
	$yQ'y^{-1}=xv^{-1} Q' vx^{-1}=xQx^{-1}$ which is contained in $J$. 
	This implies $Q'\cap y^{-1}Jy=Q'$ and so $y\in p_{Q',J,Q'}$. 
	Moreover, if $s\in K\setminus Q'$ then 
	$\ell(ys)=\ell(xv^{-1}s)=\ell(x)+\ell(v^{-1}s)= 
	\ell(x)+\ell(v^{-1})-1=\ell(y)-1$.
	Now, take $y\in p_{Q',J,Q'}$ such that $K\setminus Q'\subset D_R(y)$.
	By Lemma \ref{lemma:AB} we have 
	$\bm\Phi_{K}^+\setminus\bm\Phi_{Q'}^+\subset N(y)$ and so by
	Lemma \ref{lemma:rightdivisor} we have $\ell(yv)=\ell(y)-\ell(v)$ which 
	implies $x=yv\in {}^JW$ by Lemma \ref{lemma:compatibilita}$(a)$.
	Now, since $x\hat{Q}=yv\hat{Q}=y\hat{Q}'\subset J\subset \bm\Phi^+$, we 
	have 
	$\ell(xs)=\ell(x)+1$ for every $s\in Q$, while if $s\in K\setminus Q$ then
	$\ell(xs)=\ell(yvs)\geq\ell(y)-\ell(vs)=\ell(y)-\ell(v)+1=\ell(x)+1$.  
	Hence, $x\in {}^JW^{K}$ and so it remains to prove $K\cap x^{-1}Jx=Q$.
	Since $x\in {}^JW^K$ and $y=xv^{-1}\in {}^JW$, by Lemma 	
	\ref{lemma:compatibilita}$(c)$ we have $v^{-1}\in {}^{K\cap x^{-1}Jx}W$ 
	and then $K\cap x^{-1}Jx\subset A_L(v^{-1})\cap K=Q$. 
	Moreover, since $Q'=Q'\cap y^{-1}Jy$, we obtain $Q=Q\cap x^{-1}Jx\subset 
	K\cap 
	x^{-1}Jx$.
\end{proof}

\begin{corol}\label{coroll:riduzione}
	Let $J,K\subset S$ and $Q\subset K$. We have 
	\begin{equation}\label{eq:induzione}
	\sum_{Q\subset H\subset 
		K}\sum_{Q\subset 
		R\subset H}(-1)^{|H|-|Q|} p_{R,J,H}(t)=
	\left\{
	\begin{array}{ll}
	t^{\ell(w(K,Q'))}p_{Q',J,K}(t)& \text{if } W_K^Q \text{ is finite}\\
	0& \text{otherwise }
	\end{array}
	\right.
	\end{equation}
	where $Q'=w(K,Q)Qw(K,Q)^{-1}$.
\end{corol}

\begin{proof}
	Since for every $Q\subset H\subset K$ we have
	$\bigsqcup_{Q\subset R\subset H}p_{R,J,H}=
	\{x\in{}^JW^H\,|\,xQx^{-1}\subset J\}$,
	we obtain
	\begin{align*}
	\sum_{Q\subset H\subset 
		K}(-1)^{|H|-|Q|}\sum_{Q\subset 
		R\subset H} p_{R,J,H}(t)&=
	\sum_{Q\subset H\subset 
		K}(-1)^{|H|-|Q|}
	\sum_{\substack{x\in {}^JW\\ A_R(x)\supset H\\ xQx^{-1}\subset J}}
	t^{\ell(x)}\\
	&=\sum_{\substack{x\in {}^JW\\ xQx^{-1}\subset J}}
	\sum_{Q\subset H\subset K\cap A_R(x)}
	(-1)^{|H|-|Q|}t^{\ell(x)}
	\end{align*}
	Now, $\sum_{Q\subset H\subset K\cap A_R(x)}
	(-1)^{|H|-|Q|}$ is $1$ if $K\cap 
	A_R(x)=Q$ and $0$ otherwise.
	Hence, the left-hand size of (\ref{eq:induzione}) is the Poincaré series of 
	\[\{x\in{}^JW\,|\, 
	xQx^{-1}\subset J, K\cap A_R(x)=Q\}=\{x\in p_{Q,J,Q}\,|\,K\setminus 
	Q\subset D_R(x)\}.\]
	Now, if $W_K^Q$ is finite then by Lemma 
	\ref{lemma:uguaglianzainsiemi} this set is equal to $\{xw(K,Q')\,|\,x\in 
	p_{Q',J,K}\}$ while if $W_K^Q$ is infinite then it is empty by Lemma 
	\ref{lemma:AB} because $N(w)$ must be finite for every $w\in W$.  
\end{proof}

We remark that if in (\ref{eq:induzione}) we take $Q=J=\emptyset$ and $K=S$, we 
obtain formula (\ref{eq:Wt}) of the introduction.

\begin{prop}\label{prop:QJK}
	Let $(W,S)$ be a Coxeter system and $J,K\subset S$. 
	If $p_{H,J,H}(t)$ is a rational function for every $H\subset K$, then 
	$p_{Q,J,K}(t)$ is a rational function for every $Q\subset K$. 
\end{prop}

\begin{proof}
	We prove the assertion by induction on $|K|$. If $K=\emptyset$ then 
	$p_{\emptyset,J,\emptyset}(t)={}^JW(t)$ is rational. Now, we take $K\neq 
	\emptyset$, we suppose that 
	$p_{R,J,H}(t)$ is rational for every $H\subsetneq K$ and every $R\subset H$ 
	and we prove that $p_{Q,J,K}(t)$ is rational for every $Q\subset K$.
	We proceed by induction on $|K|-|Q|$: if $K=Q$, by hypothesis we have 
	$p_{K,J,K}(t)$ rational and so we can take $Q\subsetneq K$ and we can 
	suppose that $p_{T,J,K}(t)$ is rational for every $T\subset K$ with 
	$|T|>|Q|$. \\
	\textit{First case}: $W_K^Q$ finite and $Q'=w(K,Q)Qw(K,Q)^{-1}$ different 
	from $Q$.	
	Since $|Q|=|Q'|$ and $\ell(w(K,Q))=\ell(w(K,Q'))$, by Corollary 
	\ref{coroll:riduzione} we obtain
	\begin{align*}
	t^{\ell(w(K,Q))}p_{Q,J,K}(t)- (-1)^{|K|-|Q|}p_{Q',J,K}&=
	f_{Q'}(t)\\
	t^{\ell(w(K,Q))}p_{Q',J,K}(t) -(-1)^{|K|-|Q|}p_{Q,J,K}&=
	f_Q(t)	
	\end{align*}
	where, for $L\in\{Q,Q'\}$
	\[f_L(t)=\sum_{L\subset H\subsetneq K}\sum_{L\subset 
		R\subset H}(-1)^{|H|-|L|} p_{R,J,H}(t)+\sum_{L\subsetneq 
		R\subset K}(-1)^{|K|-|L|} p_{R,J,K}(t)\] is a rational function by 
		inductive hypothesis.
		We obtain
		\[p_{Q,J,K}(t)=\frac{t^{\ell(w(K,Q))}f_{Q'}(t)+(-1)^{|K|-|Q|}f_Q(t)}
		{t^{2\ell(w(K,Q))}-1}\]
		which is a rational function.\\
	\textit{Second case}: $W_K^Q$ finite and $Q=w(K,Q)Qw(K,Q)^{-1}$.	
	By Corollary \ref{coroll:riduzione} we have
	\[p_{Q,J,K}(t) =\frac{f_Q(t)}{t^{\ell(w_Kw_{Q})}-(-1)^{|K|-|Q|}}\]
	which is a rational function.\\
	\textit{Third case}: $W_K^Q$ infinite. By Corollary \ref{coroll:riduzione} 
	we obtain $p_{Q,J,K}(t) =(-1)^{|K|-|Q|+1}f_Q(t)$ which is a rational 
	function.
\end{proof}

\subsection{Further relations}\label{par:otherrelations}
Proposition \ref{prop:QJK} allows us to reduce the problem of studying the 
rationality of $p^{S'}_{Q,J,K}(t)$ to the case where $Q=K$.
In this paragraph, to further simplify the problem, we introduce a partition of 
$p^{S'}_{K,J,K}(t)$ and we study the Poincaré series of these subsets.

\medskip
For every $J,K\subset S'\subset S$ and $R\subset J$ we define the set
\[h^{S'}_{R,J,K}=\{x\in {}^JW_{S'}^K\,|\,xKx^{-1}=R\}=\{x\in 
{}^JW_{S'}\,|\,x\hat{K}=\hat{R}\}\] 
and we write $h^{S}_{R,J,K}=h_{R,J,K}$.
We observe that if $|R|\neq |K|$ then $h^{S'}_{R,J,K}=\emptyset$ and that 

\begin{equation}\label{eq:pKJK}
	p^{S'}_{K,J,K}= \{x\in {}^JW_{S'}^K\,|\,xKx^{-1}\subset J\} 
	=\bigsqcup_{R\subset J}h^{S'}_{R,J,K}.
\end{equation} 
In the rest of this paragraph, we show that the series $h^{S'}_{R,J,K}(t)$ 
satisfy similar relations to those proved in previous paragraphs for 
$p^{S'}_{Q,J,K}(t)$.

\begin{prop}\label{prop:razionale3}
	Let $R,J,K,J'\subset S'\subset S$ with $R\subset J\subset J'$.
	Then
	\begin{equation*}
		h^{S'}_{R,J,K}(t)=
		\sum_{R'\subset J'}h^{J'}_{R,J,R'}(t)\;h^{S'}_{R',J',K}(t).
	\end{equation*}
\end{prop}

\begin{proof}
	We can suppose $|K|=|R|$.
	In order to prove the proposition, we have to show
	\begin{equation*}
		h^{S'}_{R,J,K}=
		\bigsqcup_{R'\subset J'}
		\bigsqcup_{\substack{y\in{}^{J'}W_{S'}^{K}\\ yKy^{-1}=R'}}
		\bigsqcup_{\substack{z\in {}^{J}W_{J'}^{R'}\\ zR'z^{-1}=R}}
		\{zy\}.
	\end{equation*}
	Let $x\in h^{S'}_{R,J,K}$. By Lemma 
	\ref{lemma:razionale1} we can write in a unique way $x=zy$ with $y\in 
	{}^{J'}W_{S'}^K$ and $z\in{}^JW_{J'}^{R'}$ where $R'=J'\cap yKy^{-1}$.
	We have $W_{R'}=W_{J'}\cap yW_Ky^{-1}=z^{-1}(W_{J'}\cap xW_Kx^{-1})z 
	=z^{-1}W_Rz$.
	This implies $zR'z^{-1}=R$, $|R'|=|R|=|K|$ and so $yKy^{-1}=R'$. \\
	Now, if we take $R'\subset J'$, $y\in h^{S'}_{R',J',K}$ and $z\in 
	h^{J'}_{R,J,R'}$ then we have $x=zy\in {}^JW_{S'}^K$ by Lemma 
	\ref{lemma:compatibilita} and $xKx^{-1}=zyKy^{-1}z^{-1}=zR'z^{-1}=R$.
\end{proof}

\begin{rmk}\label{rmk:N}
	For every $J\subset S'\subset S$, consider the matrix 
	\[N_{J,S'}:=\Big(h^{S'}_{R,J,K}(t)\Big)_{R\subset J,K\subset S'}\in 
	\mathrm{Mat}_{2^{|J|}\times 2^{|S'|}}(\Z[[t]]).\]
	Proposition \ref{prop:razionale3} says that $N_{J,S'}=N_{J,J'}N_{J',S'}$ 
	for 
	every $J\subset J'\subset S'\subset S$.
\end{rmk}

\begin{lemma}\label{lemma:uguaglianzainsiemi2}
	Let $J,K\subset S$ and $R\subset J$ with $W_J^R$ finite. Let $w(J,R)$ be 
	the element of maximal length in $W_J^R$ and let $R'=w(J,R)Rw(J,R)^{-1}$. 
	Then we have 
	\[\{w(J,R)x\,|\,x\in h_{R,J,K}\}=\{y\in h_{R',R',K}\,|\, J\setminus 
	R'\subset D_L(y)\}.\]
\end{lemma}

\begin{proof}
	Let $v=w(J,R)$. We recall that $v^{-1}=w(J,R')$, $D_R(v)=J\setminus R$ and 
	$D_L(v)=J\setminus R'$.
	If we take $x\in h_{R,J,K}$  then $y=vx$ is in ${}^{R'}W^K$ by 
	Lemma \ref{lemma:compatibilita}.
	Moreover, we have $yKy^{-1}=vx Kx^{-1}v^{-1}=vRv^{-1}=R'$ and if 
	$s\in J\setminus R'$ then 
	$\ell(sy)=\ell(svx)=\ell(sv)+\ell(x)=\ell(y)-1$.
	On the other hand, take $y\in h_{R',R',K}$ such that $J\setminus 
	R'\subset D_L(y)$. Since $y^{-1}R'y=K$ and $J\setminus 
	R'\subset D_R(y^{-1})$, by Lemma \ref{lemma:AB} we have 
	$\bm\Phi_{J}^+\setminus\bm\Phi_{R'}^+\subset N(y^{-1})$  and so by
	Lemma \ref{lemma:rightdivisor} we have $\ell(y^{-1}v)=\ell(y)-\ell(v)$ 
	which implies $y^{-1}v\in {}^KW$ by Lemma \ref{lemma:compatibilita}$(a)$.
	Now, since $y^{-1}v\hat{R}=y\hat{R}'=K\subset \bm\Phi^+$, we have 
	$\ell(y^{-1}vs)=\ell(y^{-1}v)+1$ for every $s\in R$, while if $s\in 
	J\setminus R$ then
	$\ell(y^{-1}vs)\geq\ell(y^{-1})-\ell(vs)=\ell(y^{-1})-\ell(v)+1 
	=\ell(y^{-1}v)+1$.
	Hence, $y^{-1}v\in {}^KW^{J}$ and so $x=v^{-1}y\in {}^JW^{K}$.
	Finally, we have $xKx^{-1}=v^{-1}yKy^{-1}v=v^{-1}R'v=R$ and so $x\in 
	h_{R,J,K}$.
\end{proof}

\begin{corol}\label{coroll:riduzione2}
	Let $J,K\subset S$ and $R\subset J$. 
	We have 
	\begin{equation}\label{eq:induzione2}
		\sum_{R\subset H\subset J}(-1)^{|H|-|R|} h_{R,H,K}(t)=
		\left\{
		\begin{array}{ll}
		t^{\ell(w(J,R'))}h_{R',J,K}(t)& \text{if } W_J^R \text{ is finite}\\
		0& \text{otherwise }
		\end{array}
		\right.
	\end{equation}
	where $R'=w(J,R)Rw(J,R)^{-1}$.  
\end{corol}

\begin{proof}
	We have
	\begin{align*}
	\sum_{R\subset H\subset J}(-1)^{|H|-|R|} h_{R,H,K}(t)&=
	\sum_{R\subset H\subset J}(-1)^{|H|-|R|}
	\sum_{\substack{x\in W^K\\ A_L(x)\supset H\\ xKx^{-1}=R}}
	t^{\ell(x)}\\
	&=
	\sum_{\substack{x\in W^K\\ xKx^{-1}=R}}
	\sum_{R\subset H\subset J\cap A_L(x)}(-1)^{|H|-|R|}
	t^{\ell(x)}
	\end{align*}
	Now, $\sum_{R\subset H\subset J\cap A_L(x)}(-1)^{|H|-|R|}$ is $1$ if 
	$J\cap A_L(x)=R$ and $0$ otherwise.
	Hence, the left-hand size of (\ref{eq:induzione2}) is the Poincaré series 
	of 
	\[\{x\in {}^RW^K\,|\, 
	xKx^{-1}=R=J\cap A_L(x)\}=\{x\in h_{R,R,K}\,|\,J\setminus 
	R\subset D_L(x)\}.\]
	Now, if $W_J^R$ is finite then by Lemma 
	\ref{lemma:uguaglianzainsiemi2} this set is equal to $\{w(J,R')x\,|\,x\in 
	h_{R',J,K}\}$ while if $W_J^R$ is infinite then it is empty by Lemma 
	\ref{lemma:AB} because $N(w)$ must be finite for every $w\in W$. 
\end{proof}

\begin{prop}\label{prop:KJK}
	Let $(W,S)$ be a Coxeter system and $J,K\subset S$. 
	If $h_{H,H,K}(t)$ is a rational function for every $H\subset J$ with 
	$|H|=|K|$, then 
	$h_{R,J,K}(t)$ is a rational function for every $R\subset J$.
\end{prop}

\begin{proof}
	We prove the assertion by induction on $|J|$. If $J=\emptyset$ then 
	$h_{\emptyset,\emptyset,K}(t)$ is $W(t)$ if 
	$K=\emptyset$ and $0$ otherwise. 
	Now, we take $J\neq \emptyset$, we suppose that 
	$h_{T,H,K}(t)$ is rational for every $H\subsetneq J$ and every $T\subset H$ 
	and we prove that $h_{R,J,K}(t)$ is rational for every $R\subset J$.
	If $R=J$, by hypothesis $h_{J,J,K}(t)$ is rational while if $R\subsetneq J$ 
	we can use Corollary \ref{coroll:riduzione2}.\\
	\textit{First case:} $W_J^R$ finite and $R'=w(J,R)Rw(J,R)^{-1}$ different 
	from $R$.
	Since $|R|=|R'|$ and $\ell(w(J,R))=\ell(w(J,R'))$, by Corollary 
	\ref{coroll:riduzione2} we obtain
	\begin{align*}
	t^{\ell(w(J,R))}h_{R,J,K}(t)- (-1)^{|J|-|R|}h_{R',J,K}(t)&=
	f_{R'}(t)\\
	t^{\ell(w(J,R))}h_{R',J,K}(t)- (-1)^{|J|-|R|}h_{R',J,K}(t)&=
	f_{R}(t)\\	
	\end{align*}
	where, for $L\in\{R,R'\}$
	\[f_L(t)=\sum_{L\subset H\subsetneq J}(-1)^{|H|-|L|} h_{L,H,K}(t)\] is a 
	rational function by inductive hypothesis.
	We obtain
	\[h_{R,J,K}(t)=\frac{t^{\ell(w(J,R))}f_{R'}(t)+(-1)^{|J|-|R|}f_R(t)}
	{t^{2\ell(w_Jw_{R})}-1}\]
	which is a rational function.\\
	\textit{Second case}: $W_J^R$ finite and $R=w(J,R)Rw(J,R)^{-1}$.
	By Corollary \ref{coroll:riduzione2} we have
	\[h_{R,J,K}(t) 
	=\frac{f_R(t)}{t^{\ell(w(J,R))}-(-1)^{|J|-|R|}}\]
	which is a rational function.\\
	\textit{Third case}: $W_J^R$ infinite.
	By Corollary \ref{coroll:riduzione2} 
	we obtain $h_{R,J,K}(t) =(-1)^{|J|-|R|+1}f_R(t)$ which is a rational 
	function.
\end{proof}

Proposition \ref{prop:KJK} allows us to reduce the problem of studying the 
rationality of $p^{S'}_{K,J,K}(t)$ to the case where $|J|=|K|$.
Actually, given $J,K\subset S$, if there exists $x\in W$ such that  
$x\hat{K}=\hat{J}$ then  $|J|=|K|$ and $x\in {}^JW^K$. 
Hence, we have
\[p_{K,J,K}=h_{J,J,K}=\{x\in W\,|\, xKx^{-1}=J\}.\]
We observe that if $J=K$ then $p_{J,J,J}$ is a group and we have 
$N_W(W_J)=W_J\rtimes p_{J,J,J}$.
These sets are well studied in \cite{D82} (see also \cite{BH,B98}): let's 
summarize their results.

\begin{prop}\label{prop:Deod1}
	Let $R\subset J\subset S$ such that $W_J^R$ is finite. Then $R$ contains 
	every non-spherical component of $J$. 
\end{prop}

\begin{proof}
	It is a consequence of the proof of \cite[Proposition 4.2]{D82} (see also 
	\cite[Proposition 2.2]{BH}).
\end{proof}

\begin{rmk}\label{rmk:Deod}
	Given $H\subset S$ and $s\in S\setminus H$, we recall that 
	$v:=w(H\cup\{s\},H)$ 
	exists if $W_{H\cup\{s\}}^H$ is finite.
	In this case, if $H'\cup\{s\}$ is the component of $H\cup\{s\}$ containing 
	$s$, 
	then it is spherical by Proposition \ref{prop:Deod1} and we have 
	$v=w(H'\cup\{s\},H')\in W_{H'\cup\{s\}}$.
	Hence, since  $(H'\cup\{s\}) \perp (H\setminus H')$, we have $v\in 
	W_{(H\setminus H')^\perp}$ 
	and in particular $v$ fixes all not-spherical components of $H$.
\end{rmk}

\begin{prop}[Proposition 5.5 of \cite{D82}]\label{prop:Deod2}
	Let $x\in W$ and $J,K\subset S$ such that $xKx^{-1}=J$. 
	Then there exists $s_1,\dots, s_m\in S$ and $J_0,J_1,\dots,J_m\subset 
	S$ such that 
	\begin{enumerate}[(i)]
		\item $J_0=K$ and $J_m=J$;
		\item $s_i\notin J_{i-1}$ and $w_i:=w(J_{i-1}\cup \{s_i\},J_{i-1})$ 
		exists 		for every $i=1,\dots, m$;
		\item $J_i=w_iJ_{i-1}w_i^{-1}$;
		\item $x=w_m\cdots w_1$ and $\ell(x)=\ell(w_1)+\cdots+\ell(w_m)$.
	\end{enumerate}
\end{prop}

\begin{rmk}\label{rmk:Deod3}
	Let $x\in W$ and $J,K\subset S$ be such that $xKx^{-1}=J$. 
	Let $J_0,J_1,\dots,J_m$ be as in Proposition \ref{prop:Deod2}.
	By Remark \ref{rmk:Deod}, for every $i\in\{1,\dots,m\}$ all not-spherical 
	components of $J_{i-1}$ and $J_i$ coincide and, if $H$ is the union of all 
	these non-spherical components, then $x\in W_{H^\perp}$.
\end{rmk}

\begin{rmk}\label{rmk:K'J'K'}
Let $J,K\subset S$ such that $|J|=|K|$ and $p_{K,J,K}\neq\emptyset$.
Let $H$ be the union of all non-spherical components of $J$, that by Remark
\ref{rmk:Deod3} coincide with that of $K$.  
Then $J'=J\setminus H$ and $K'=K\setminus H$ are the unions of spherical 
components of $J$ and $K$, respectively. 
Given $x\in W$, Remark \ref{rmk:Deod3} implies that $xKx^{-1}=J$ if and 
only if $x\in W_{H^\perp}$ and $xK'x^{-1}=J'$, which means 
$p_{K,J,K}=p^{H^\perp}_{K',J',K'}.$
\end{rmk}

\begin{lemma}
	Let $K\subset S$ such that every irreducible component of $K$ is 
	not-spherical. 
	Then  the series $h_{R,J,K}(t)$ and $p_{K,J,K}(t)$ are rational function 
	for every $R\subset J\subset S$.
\end{lemma}

\begin{proof}
	Given $H\subset S$ such that $|H|=|K|$, by Remark \ref{rmk:K'J'K'} we have
	$p_{K,H,K}\neq \emptyset$ if and only if $H=K$ and 
	$p_{K,K,K}(t)=W_{K^\perp}(t)$. Hence, by Proposition 
	\ref{prop:KJK} and (\ref{eq:pKJK}), the series $h_{R,J,K}(t)$ and 
	$p_{K,J,K}(t)$ are rational function for every 
	$R\subset J\subset S$.
\end{proof}

\begin{lemma}
	Let $J,K\subset S$ such that $W_K$ is a maximal finite parabolic subgroup 
	of $W$. Then $p_{K,J,K}(t)=1$ if $K\subset J$ and $0$ otherwise.
\end{lemma}

\begin{proof}
	Let $x\in W$ such that $xKx^{-1}\subset S$. If $x\neq 1$, by 
	Proposition \ref{prop:Deod2} there should be an element $s_1\notin K$ such 
	that $w(K\cup\{s_1\},K)$ exists, but this is not possible since 
	$W_{K\cup\{s\}}$ is not finite.
	Hence, $x=1$ and so $p_{K,J,K}=\emptyset$ if $K\not\subset J$ and 
	$p_{K,J,K}=\{1\}$ if $K\subset J$.
\end{proof}

\subsection{The main theorem}
Finally, we summarize the previous results by stating the main theorem.
\begin{teor}\label{teor:equivalenze}
	Let $(W,S)$ be a Coxeter system. 
	The following are equivalent.
	\begin{enumerate}[(i)]
		\item $p^{S'}_{Q,J,K}(t)$ is a rational function for every 
		$Q,J,K\subset S'\subset S$;
		\item $p^{S'}_{Q,J,K}(t)$ is a rational function for every 
		$Q,J,K\subset S'\subset S$ with $|K|=|S'|-1$;
		\item $p^{S'}_{K,J,K}$ is a rational function for every $J,K\subset 
		S'\subset S$;
		\item $h^{S'}_{R,J,K}(t)$ is a rational function for every $J,K\subset 
		S'\subset S$ and $R\subset J$;
		\item $h^{S'}_{R,J,K}(t)$ is a rational function for every $J,K\subset 
		S'\subset S$, $R\subset J$ with $|J|=|S'|-1$;
		\item $p^{S'}_{K,J,K}(t)=h^{S'}_{J,J,K}(t)$ is a rational function for 
		every $J,K\subset S'\subset S$ with $|J|=|K|$;
		\item $p^{S'}_{K,J,K}(t)$ is a rational function for 
		every $J,K\subset S'\subset S$ with $|J|=|K|$ and $J,K$ spherical. 
	\end{enumerate}
Moreover, if these conditions hold, then ${}^JW_{S'}^K(t)$ is a rational 
function for every $J,K\subset S'\subset S$.
\end{teor}

\begin{proof} We have $(iii)\Rightarrow (i)$ by Proposition \ref{prop:QJK}, 
$(vi)\Rightarrow (iv)$ by Proposition \ref{prop:KJK} 
	 and 	
	$(vii)\Rightarrow (vi)$ by Remark \ref{rmk:K'J'K'}.
	Hence, $(i),(iii),(iv),(vi),(vii)$ are equivalent since $(i)\Rightarrow 
	(vii)$ and	$(iv)\Rightarrow (iii)$ are obvious.
	Furthermore, $(i)\Rightarrow (ii)$ and $(iv)\Rightarrow (v)$ are obvious and
	so the following implications complete the proof.
	\begin{itemize}
		\item[$(ii)\Rightarrow (i)$]  Let $K\subset S'\subset S$ and let 
		$K=K_0\subset K_1\subset \cdots \subset K_r=S'$ with 
		$|K_i|=|K_{i-1}|+1$ for every $i\in\{1,\dots,r\}$.
		By hypothesis, the matrices $M_{K_{i-1},K_i}$ with $i\in\{1,\dots,r\}$ 
		have coefficients in $\Q(t)$ and by Remark \ref{rmk:M} we have 
		$M_{K,S'}=\prod_{i=1}^rM_{K_{i-1},K_i}$ which therefore has 
		coefficients in $\Q(t)$.
		\item[$(v)\Rightarrow (iv)$] Let $J\subset S'\subset S$ and let 
		$J=J_0\subset J_1\subset \cdots \subset J_r=S'$ with 
		$|J_i|=|J_{i-1}|+1$ for every $i\in\{1,\dots,r\}$.
		By hypothesis, the matrices $N_{J_{i-1},J_i}$ with $i\in\{1,\dots,r\}$ 
		have coefficients in $\Q(t)$ and by Remark \ref{rmk:N} we have 
		$N_{J,S'}=\prod_{i=1}^rN_{J_{i-1},J_i}$ 
		which therefore has coefficients in $\Q(t)$.\qedhere
	\end{itemize}
\end{proof}

\subsection{About the normalizer of a finite parabolic 
subgroup}\label{sec:normalizer}
In this paragraph, we discuss about the series appearing in $(vii)$ of Theorem 
\ref{teor:equivalenze}, that is the Poincaré series of 
\[p_{K,J,K}=\{x\in W\,|\, xKx^{-1}=J\}\]
where $J,K\subset S$ are spherical and $|J|=|K|$.

\medskip
As observed at the end of paragraph \ref{par:otherrelations}, these sets are 
well studied by in \cite{D82,BH,B98}.
In particular, if $J=K$ then 
$N_J:=p_{J,J,J}$ is a group, $N_W(W_J)=W_J\rtimes N_J$ and 
$N_W(W_J)(t)=W_J(t)N_J(t)$.
Moreover, the multiplication in $W$ induces a transitive left action of $N_J$ 
and a transitive right action of $N_K$ on $p_{K,J,K}$.
Hence, the study of the rationality of $p_{K,J,K}(t)$ is closely related to the 
study of the rationality of $N_J(t)$.
We know various description of $N_J$.
\begin{itemize}
	\item By \cite[Theorem B]{BH}, the group $N_J$ is the semidirect product 
	of a Coxeter group $\widetilde{N}_J$ and a group $M_J$ which acts on  
	$\widetilde{N}_J$ by permuting its Coxeter generating set.
	A similar description is given in \cite{B98}.
	\item By \cite[Theorem 5.7]{Nuida}, if 
	$\bm\Phi^{\perp J}=\{\gamma\in\bm\Phi\,|\, \gamma 
	\text{ is orthogonal to every }\alpha\in \hat{J}\}$ then	
	$N_J$ is the semidirect product of 	the reflection subgroup of $W$ 
	generated by 
	$\{ws_\alpha w^{-1}\,|\, w\alpha\in\bm\Phi ^{\perp J}\}$ and 
	the group 
	$\{w\in N_J\,|\, w(\bm\Phi^{\perp J})^+=(\bm\Phi^{\perp J})^+\}$.
\end{itemize}
Unfortunately, these decompositions of $N_J$ do not allow, for now, to 
determine whether its Poincaré series is rational or not.
Actually, both of these decompositions do not respect the length function of 
$W$, which means that the Poincaré series of $N_J$ is not the product of the 
Poincaré series of its  factors.

\medskip
Another way to prove the rationality of $p_{K,J,K}(t)$ would be to use 
Proposition \ref{prop:Deod2} which states that every $x\in p_{K,J,K}$ can be 
written as $x=w_m\cdots w_1$ where every $w_i$ is of the form 
$w(L\cup\{\alpha\},L)$ and $\ell(x)=\ell(w_1)+\cdots+\ell(w_m)$.

\section{Generating functions of rational cones}\label{section:coni}
In this section we introduce some definitions and properties about generating 
functions of rational simplicial cones (for more details see \cite[Chapter 
3]{BR}). We will use these results in section \ref{section:affine} to prove the 
rationality of ${}^JW^K(t)$ when $W$ is an affine Weyl group.

\medskip
Let $n\in\N$ and $X$ be a subset of $\R^n$.
We define the multivariate generating function of $X$ by
\[\sigma_X(\mathbf{z})=\sigma_X(z_1,\dots,z_n):=
\sum_{\mathbf{m}\in X\cap\Z^n}\mathbf{z}^\mathbf{m}\in \Z[[\mathbf{z}]]\]
where $\mathbf{z}^{\mathbf{m}}=z_1^{m_1}\cdots z_n^{m_n}$ for 
$\mathbf{m}=(m_1,\dots,m_n)\in\Z^n$.

\medskip
Let $d,n\in\N$ with $d\leq n$. A \emph{rational simplicial} $d$-cone 
in $\R^n$ is a set of the form 
\[\mathcal{C}=\mathrm{cone}(\mathbf{v}_1,\dots,\mathbf{v}_d) 
=\{\lambda_1\mathbf{v}_1+\cdots+\lambda_d\mathbf{v}_d\in\R^n\,|\, 
\lambda_1,\dots,\lambda_d\in\R_{\geq 0}\}\] 
where $\mathbf{v}_1,\dots,\mathbf{v}_d\in\Z^n$ are $\R$-linearly 
indipendent vectors. 
The dimension of $\mathcal{C}$ is $d$ and we denote it by $\dim(\mathcal{C})$.
The fundamental 
parallelepiped of $\mathcal{C}$ is the set 
\[\Pi(\mathcal{C})=\Pi(\mathbf{v}_1,\dots,\mathbf{v}_d) 
=\{\lambda_1\mathbf{v}_1+\cdots+\lambda_d\mathbf{v}_d\in\R^n
\,|\, 0\leq\lambda_1,\dots,\lambda_d<1\},\]
while the open cone associated to $\mathcal{C}$ is 
$\mathcal{C}^o=\{\lambda_1\mathbf{v}_1+\cdots+\lambda_d\mathbf{v}_d\in\R^n
\,|\, \lambda_1,\dots,\lambda_d\in\R_{>0}\}.$
If $d=0$, we put $\mathcal{C}=\Pi(\mathcal{C})=\mathcal{C}^o=\{(0,\dots,0)\}$.
We say that a rational simplicial cone $\mathcal{D}$ is contained in 
$\mathcal{C}$, and we write $\mathcal{D}\subset\mathcal{C}$, if 
$\mathcal{D}=\mathrm{cone}(X)$ with $X\subset 
\{\mathbf{v}_1,\dots,\mathbf{v}_d \}$.

\begin{lemma}\label{lemma:cono}
	For every rational simplicial cone we have
	$\displaystyle{		
	\mathcal{C}=\bigsqcup_{\mathcal{D}\subset\mathcal{C}}\mathcal{D}^o.}$
\end{lemma}

\begin{proof}
	For every 
	$\mathbf{v}=\lambda_1\mathbf{v}_1+\cdots+\lambda_d\mathbf{v}_d\in\mathcal{C}$,
	let $\mathcal{D}_{\mathbf{v}}= 
	\mathrm{cone}(\{\mathbf{v}_i\,|\,\lambda_i\neq 0\})$.
	Then $\mathbf{v}\in\mathcal{D}_{\mathbf{v}}^o$ and so 
	$\mathcal{C}=\bigcup_{\mathcal{D}\subset\mathcal{C}}\mathcal{D}^o$.
	The union is disjoint because $\mathbf{v}_1,\dots,\mathbf{v}_d$ are 
	linearly independent.
\end{proof}

\begin{prop}\label{prop:cono}
	Let $\mathcal{C}=\mathrm{cone}(\mathbf{v}_1,\dots,\mathbf{v}_d) $ be a 
	rational simplicial $d$-cone.
	\begin{enumerate}[(a)]
		\item The series $\sigma_\mathcal{C}(\mathbf{z})$ is a rational 
		function and 
		$\displaystyle{\sigma_\mathcal{C}(\mathbf{z})= 
		\sigma_{\Pi(\mathcal{C})}(\mathbf{z})
		\prod_{i=1}^{d}(1-\mathbf{z}^{\mathbf{v}_i})^{-1}.}$
		\item The series $\sigma_{\mathcal{C}^o}(\mathbf{z})$ is a rational 
		function and 
		$\displaystyle{	\sigma_{\mathcal{C}^o}(\mathbf{z})=
		\sum_{\mathcal{D}\subseteq\mathcal{C}}
		(-1)^{\dim(\mathcal{C})-\dim(\mathcal{D})}
		\sigma_{\mathcal{D}}(\mathbf{z})}.$
	\end{enumerate}
\end{prop}

\begin{proof}
	For $\lambda\in\R$ we write $\lfloor \lambda\rfloor$ for 
	its integer part and $\{\lambda\}=\lambda-\lfloor \lambda\rfloor$ for its 
	fractional part. 	
	Since $\mathbf{v}_1,\dots,\mathbf{v}_d$ are linearly independent, every 
	$\mathbf{m}\in\mathcal{C}\cap\Z^n$ can be written uniquely as 
	$\mathbf{m}=\lambda_1\mathbf{v}_1+\cdots+\lambda_d\mathbf{v}_d 
	=\lfloor\lambda_1\rfloor\mathbf{v}_1+\cdots+\lfloor\lambda_d\rfloor\mathbf{v}_d
	+\{\lambda_1\}\mathbf{v}_1+\cdots+\{\lambda_d\}\mathbf{v}_d$ with 
	$\lambda_1,\dots,\lambda_d\geq 0$. 
	In particular, $\{\lambda_1\}\mathbf{v}_1+\cdots+\{\lambda_d\}\mathbf{v}_d$
	belongs to $\Pi(\mathcal{C})\cap\Z^n$ since $\mathbf{m}$ and 
	$\lfloor\lambda_1\rfloor\mathbf{v}_1+\cdots+\lfloor\lambda_d\rfloor\mathbf{v}_d$
	are in $\Z^n$.
	Hence, every $\mathbf{m}\in\mathcal{C}\cap\Z^n$ can be written uniquely as 
	$\mathbf{m}=\mathbf{p}+k_1\mathbf{v}_1+\cdots+k_d\mathbf{v}_d$ with 
	$\mathbf{p}\in\Pi(\mathcal{C})\cap\Z^n$ and $k_1,\dots,k_d\in\N$ and so 
	\[\sigma_{\Pi(\mathcal{C})}(\mathbf{z})
	\prod_{i=1}^d(1-\mathbf{z}^{\mathbf{v}_i})^{-1}
	=\left(\sum_{\mathbf{p}\in\Pi(\mathcal{C})\cap\Z^n}\mathbf{z}^{\mathbf{p}}\right)
	\left(\sum_{k_1\in\N}\mathbf{z}^{k_1\mathbf{v_1}}\right)\cdots
	\left(\sum_{k_d\in\N}\mathbf{z}^{k_d\mathbf{v_d}}\right)
	=\sum_{\mathbf{m}\in\mathcal{C}\cap\Z^n}\mathbf{z}^{\mathbf{m}},\]
	which proves $(a)$.
	Now, by Lemma \ref{lemma:cono} we have 
	$\sigma_{\mathcal{C}}(\mathbf{z})=
	\sum_{\mathcal{D}\subset\mathcal{C}}
	\sigma_{\mathcal{D}^o}(\mathbf{z})$
	and using M\"obius inversion formula we obtain $(b)$.
\end{proof}

\begin{rmk}\label{rmk:vettoriridotti}
	The fundamental parallelepiped $\Pi(\mathcal{C})$ of a rational simplicial 
	cone $\mathcal{C}$ depends on the choice of vectors 
	$\mathbf{v}_1,\dots,\mathbf{v}_d$ who define $\mathcal{C}$.
	On the other hand, the series $\sigma_{\mathcal{C}}(\mathbf{z})$ and 
	$\sigma_{\mathcal{C}^o}(\mathbf{z})$ do not depend on this choice, so we 
	can choose each $\mathbf{v}_i$ so that the greatest common divisor of 
	its coefficients is 1.
\end{rmk}

\section{Affine Weyl groups}\label{section:affine}
In this section we focus on the case when $W$ is an affine Weyl group. The 
objective is to prove that, in some special cases, the series $p_{Q,J,K}(t)$ 
and ${}^JW^K(t)$ are rational functions.

\medskip
Even if we have proved Theorem \ref{teor:equivalenze} for general Coxeter 
groups, we will not use it in the case of affine Weyl group. 
However, we will use Proposition \ref{prop:razionale2} and results of 
section \ref{section:coni}. 
In paragraphs \ref{par:root} and \ref{par:affineWeyl} we recall the definitions 
of finite root systems and of affine Weyl groups, while in paragraph 
\ref{par:rationalityaffine} we state and prove the main theorem of this section 
(Theorem \ref{teor:riassunto}).

\subsection{Root systems}\label{par:root}
Let $V$ be a finite dimensional real vector space, let $V^{\vee}$ be its dual 
vector space and let $\langle\cdot,\cdot\rangle:V\times V^{\vee}\rightarrow 
\mathbb{R}$ be the natural pairing.
A subset $\bm\Phi$ of $V$ is a \emph{finite root system} in $V$ if it is 
finite, 
it spans $V$, $0\notin\bm\Phi$ and for every $\alpha\in\bm\Phi$ there 
exists a unique $\alpha^{\vee}\in V^{\vee}$ such that $\langle 
\alpha,\alpha^{\vee}\rangle=2$, 
$\langle\bm\Phi,\alpha^{\vee}\rangle\subset\mathbb{Z}$ and 
$s_\alpha:x\mapsto x-\langle x,\alpha^{\vee}\rangle\alpha$ maps $\bm\Phi$ 
into $\bm\Phi$. 	
The elements of $\bm\Phi$ are called \emph{roots}.
We say that a finite root system $\bm\Phi$ in $V$ is \emph{reduced} if 
$\mathbb{Z}\alpha\cap R=\{\pm\alpha\}$ for every $\alpha\in\bm\Phi$ and it is 
\emph{indecomposable} if it is not the disjoint union of two non-empty root 
systems. 

\medskip
We fix a reduced finite root systems $\bm\Phi$ in $V$.
The \emph{Weyl group} $\mathcal{W}=\mathcal{W}(\bm\Phi)$ of $\bm\Phi$ is 
the subgroup of $\mathrm{Aut}(V)$ generated by the reflections $s_\alpha$ 
for $\alpha\in\bm\Phi$.

\medskip
A subset $\Sigma$ of $\bm\Phi$ is a \emph{base} for $\bm\Phi$ if it is a 
base for $V$ and if every $\alpha\in\bm\Phi$ is a linear combination of 
elements of $\Sigma$ with all non-negative or all non-positive integer 
coefficients.
The elements of $\Sigma$ are called \emph{simple roots}.
Let $\bm\Phi^{+}$ be the set of roots which are linear 
combination of simple roots with all non-negative coefficients and let 
$\bm\Phi^-=\{-\alpha\,|\,\alpha\in\bm\Phi^+\}$.
The elements of $\bm\Phi^{+}$ are called \emph{positive roots} while the 
elements of $\bm\Phi^{-}$ are called \emph{negative roots}. 
Clearly we have $\bm\Phi=\bm\Phi^{+}\sqcup\bm\Phi^{-}$.
The group $\mathcal{W}(\bm\Phi)$ acts 
faithfully on $\bm\Phi$, it is generated by the reflections 
$\mathcal{S}=\mathcal{S}(\Sigma)=\{s_\alpha\,|\,\alpha\in\Sigma\}$ and 
$(\mathcal{W},\mathcal{S})$ is a finite Coxeter system.
Moreover, there exists a unique root 
$\widetilde{\alpha}=\sum_{\alpha\in\Sigma}n_\alpha\alpha\in\bm\Phi^+$, called 
\emph{highest root}, such that for every 
$\beta=\sum_{\alpha\in\Sigma}m_\alpha\alpha\in\bm\Phi$ we have $n_\alpha\geq 
m_\alpha$ for every $\alpha\in\Sigma$.

\medskip
Let $\bm\Phi^{\vee}$ be the set of $\alpha^{\vee}$ for $\alpha\in\bm\Phi$. 
Then $\bm\Phi^{\vee}$ is a finite root system in $V^{\vee}$, called 
\emph{inverse (or dual) root system}, and its elements are called 
\emph{coroots}.
Moreover, if $\Sigma$ is a base for $\bm\Phi$ then $\Sigma^{\vee}=\{\alpha^{\vee}\in\bm\Phi^{\vee}\,|\,\alpha\in \Sigma\}$ is a base for $\bm\Phi^{\vee}$.

\subsection{Affine Weyl groups}\label{par:affineWeyl}
We fix an irreducible reduced finite root system $\bm\Phi$ in $V$ and a base 
$\Sigma$ 
for it.
We denote by $\mathrm{Aff}(V^{\vee})$ the group of affine transformations of 
the vector space $V^{\vee}$.

\medskip
We define $\bm\Phi_{\mathrm{aff}}:=\bm\Phi\times \Z$ and we call \emph{affine 
roots} its elements.
 The \emph{affine Weyl group} of $\bm\Phi$ is the subgroup 
 $\widetilde{\mathcal{W}}=\widetilde{\mathcal{W}}(\bm\Phi)$ of 
 $\mathrm{Aff}(V^{\vee})$ generated by the set of \emph{affine reflections}
\[s_{\alpha,k}:x\mapsto 
x-(\langle\alpha,x\rangle-k)\alpha^{\vee}\]
with $(\alpha,k)\in\bm\Phi_{\mathrm{aff}}$.
If $\widetilde{\alpha}$ is the highest root of $\mathcal{W}(\bm\Phi)$, then 
the set 
\begin{equation}\label{eq:generatoriWtilda}
\widetilde{\mathcal{S}}=\widetilde{\mathcal{S}}(\Sigma)= 
\{s_{\alpha,0}\,|\,\alpha\in\Sigma\}\cup 
\{s_{-\widetilde{\alpha},1}\}
\end{equation} 
generates $\widetilde{\mathcal{W}}$ and 
$(\widetilde{\mathcal{W}},\widetilde{\mathcal{S}})$ is a Coxeter system (see 
\cite[\S4.6]{Hum}).
We identify $\mathcal{S}$ with $\{s_{\alpha,0}\,|\,\alpha\in\Sigma\}$ and 
$\mathcal{W}(\bm\Phi)$ with the subgroup of $\widetilde{\mathcal{W}}$ generated 
by these elements.

\medskip
Let $L^{\vee}$ be the coroot lattice, i.e. the $\mathbb{Z}$-span of 
$\bm\Phi^{\vee}$. The abelian subgroup 
\[\Lambda=\{t(v):x\mapsto x+v\,|\,v\in L^{\vee}\}\] 
of $\mathrm{Aff}(V^{\vee})$ is contained in
$\widetilde{\mathcal{W}}$ since $s_{\alpha,k}\circ 
s_{\alpha,0}:x\mapsto x+k\alpha^{\vee}$ for every $\alpha\in\bm\Phi$ and 
$k\in\mathbb{Z}$.
Moreover, for every $\alpha\in\Sigma$ and every $v\in L^{\vee}$ we have 
$s_{\alpha,0}t(v)s_{\alpha,0}=t(s_{\alpha,0}(v))$ and so $\Lambda$ is normal in 
$\widetilde{\mathcal{W}}$.
By \cite[Proposition 4.2]{Hum}, the group $\widetilde{\mathcal{W}}$ is the 
semidirect product of $\mathcal{W}$ and $\Lambda$.
	
\medskip
Let $\mathcal{H}$ be the collection of the hyperplanes 
$H_{\alpha,k}=\{x\in V^{\vee}\,|\,\langle \alpha,x\rangle=k\}$
with $(\alpha,k)\in \bm\Phi_\mathrm{aff}$ and let $\mathcal{A}$ be the 
collection of connected components, called \emph{alcoves}, of 
$V^{\vee}\setminus 
\bigcup_{H\in\mathcal{H}}H$.
The group $\widetilde{\mathcal{W}}$ acts faithfully on $\bm\Phi_\mathrm{aff}$ 
and transitively on $\mathcal{A}$ (see \cite[\S4.3]{Hum}).
Since we have fixed a base $\Sigma$ for $\bm\Phi$, we call \emph{fundamental 
alcove} the set
$A_0=\{x\in V^{\vee}\,|\,\langle\alpha,x\rangle >0\;\forall 
\alpha\in\Sigma \text{ and } \langle\widetilde{\alpha},x\rangle<1\}.$
Let 
$\bm\Phi_{\mathrm{aff}}^+=\{(\alpha,k)\in\bm\Phi_{\mathrm{aff}}\,|\, 
\langle\alpha,x\rangle+k>0 \;\forall x\in A_0\}$
be the set of \emph{affine postive roots}. 
In particular (see \cite[\S2]{Mac}), we have
$\bm\Phi_{\mathrm{aff}}^+=\{(\alpha,n+\chi(\alpha))\,|\,\alpha\in\bm\Phi, 
n\in\mathbb{N}\}$
where $\chi$ is the characteristic function of $\bm\Phi^-$ (i.e. 
$\chi(\alpha)=0$ if $\alpha\in\bm\Phi^+$ and $\chi(\alpha)=1$ if 
$\alpha\in\bm\Phi^-$). Let 
$\bm\Phi_{\mathrm{aff}}^-=\bm\Phi_{\mathrm{aff}}\setminus\bm\Phi_{\mathrm{aff}}^+$
 be the set of \emph{affine negative roots}. By \cite[\S2]{Mac} and 
 \cite[\S4]{Hum} we have the following lemma.

\begin{lemma}\label{lemma:length}
The length of $x\in \widetilde{\mathcal{W}}$ can be 
equivalently computed as:
\begin{enumerate}
	\item the length of a reduced expression $x=s_1\cdots s_n$ with 
	$s_i\in\widetilde{\mathcal{S}}$;
	\item the cardinality of $\bm\Phi_{\mathrm{aff}}^+\cap 
	x^{-1}\bm\Phi_{\mathrm{aff}}^-$;
	\item the cardinality of $\{H\in\mathcal{H}\,|\,H \text{ separates 
	}A_0 \text{ and }xA_0\}$;
	\item $\sum_{\alpha\in\bm\Phi^+}|\langle\alpha,v\rangle+\chi(w\alpha)|$
	if $x=wt(v)$ with $w\in\mathcal{W}$ and $v\in L^{\vee}$.
\end{enumerate}
\end{lemma}

\subsection{Rationality of ${}^J\widetilde{\mathcal{W}}^K$}
\label{par:rationalityaffine}
In this paragraph, initially we show that the rationality of $p_{Q,J,K}(t)$ 
with $Q,J,K\subset \mathcal{S}$ depends on the rationality of 
$p_{Q,\mathcal{S},\mathcal{S}}(t)$ with $Q\subset \mathcal{S}$ and after we 
prove the rationality of the last using the results of section 
\ref{section:coni}.
We recall that every proper parabolic subgroup of $\widetilde{\mathcal{W}}$ is 
finite and so $p^{\mathcal{S}'}_{Q,J,K}(t)$ is a polynomial for every 
$Q,J,K\subset \mathcal{S}'\subsetneq \mathcal{S}$.

\medskip
First of all, we want to understand,  using the 
isomorphism $\widetilde{\mathcal{W}}\cong \mathcal{W}\ltimes \Lambda$, how the 
elements of 
${}^\mathcal{S}\widetilde{\mathcal{W}}^\mathcal{S}$ looks like and which 
properties they satisfy (Proposition \ref{prop:pQSS} and Lemma 
\ref{lemma:SWS}).

\begin{lemma}\label{lemma:costanteclassi}
	Let $\lambda\in\Lambda$. 
	\begin{enumerate}[(a)]
		\item $\ell$ is constant on the 
		$\mathcal{W}$-conjugacy class $\{w\lambda w^{-1}\,|\,w\in 
		\mathcal{W}\}$ of $\lambda$;
		\item $C_\mathcal{W}(\lambda)=\mathcal{W}\cap 
		\lambda^{-1}\mathcal{W}\lambda$ and 
		$C_\mathcal{S}(\lambda)=\mathcal{S}\cap 
		\lambda^{-1}\mathcal{S}\lambda$;
		\item $C_\mathcal{S}(\lambda)\subset A_L(\lambda)$;
		\item If $\mathcal{S}\subset A_R(\lambda)$ then 
		$A_L(\lambda)\cap\mathcal{S}=C_\mathcal{S}(\lambda)$.
	\end{enumerate}
\end{lemma}

\begin{proof} Let $\lambda=t(v)$ with $v\in L^{\vee}$.
	\begin{enumerate}[$(a)$]
		\item Let $s=s_{\beta,0}\in \mathcal{W}$. 
		Using 4 of Lemma \ref{lemma:length} we have $\ell(s 
		\lambda s)= 
		\ell(t(sv))=\sum_{\alpha\in\bm\Phi^+}|\langle\alpha,sv\rangle|=
		\sum_{\alpha\in\bm\Phi^+}|\langle s\alpha,v\rangle|=	
		\sum_{\alpha\neq \beta}|\langle \alpha,v\rangle|+|\langle 
		-\beta,v\rangle|=\ell(\lambda)$ since
		$s$ fixes $\bm\Phi^+\setminus\{\beta\}$ and 
		$s(\beta)=-\beta$.
		\item If $w=\lambda^{-1} w'\lambda\in 
		\mathcal{W}\cap \lambda^{-1} \mathcal{W}\lambda$ then, applying the 
		natural projection $\pi:\widetilde{\mathcal{W}}\rightarrow 
		\mathcal{W}$, we obtain $w=\pi(w)=\pi(\lambda^{-1} w'\lambda)=w'$ and 
		so $w\in C_\mathcal{W}(\lambda)$. On the other hand if $w''\in 
		C_\mathcal{W}(\lambda)$ then $\lambda^{-1} w''\lambda=w''$ and so  
		$w''\in \mathcal{W}\cap \lambda^{-1} \mathcal{W}\lambda$.
		\item 	Let $s=s_{\alpha,0}\in 
		C_\mathcal{S}(\lambda)$.
		Since $s\lambda s=\lambda$ we have $\langle \alpha,v\rangle=0$. This 
		implies that $A_0$ and $\lambda A_0$ lie on the same side of 
		$H_{\alpha,0}$.  Hence, $H_{\alpha,0}$ separates $A_0$ and $s\lambda 
		A_0$ and so $\ell(s\lambda)=\ell(\lambda)+1$ by 3 of Lemma 
		\ref{lemma:length}.
		\item Suppose by 
		contradiction that exists $s\in 
		A_L(\lambda)\cap\mathcal{S}$ which is not in $C_\mathcal{S}(\lambda)$.
		Let $J=C_\mathcal{S}(\lambda)\cup\{s\}$. Then 
		$\lambda\in{}^J\widetilde{\mathcal{W}}^\mathcal{S}$ and 
		$\mathcal{S}\cap \lambda^{-1}J\lambda=\mathcal{S}\cap 
		(C_\mathcal{S}(\lambda)\cup\{\lambda^{-1}s\lambda\})= 
		C_\mathcal{S}(\lambda)$ since $\lambda^{-1}s\lambda\notin \mathcal{S}$ 
		by $(b)$. 
		Hence, by Lemma \ref{lemma:decwtilde} we have $\ell(s\lambda 
		s)=\ell(\lambda)+2$ which is in contradiction with $(a)$.  
		This implies $A_L(\lambda)\cap\mathcal{S}\subset 
		C_\mathcal{S}(\lambda)$ and so the equality by $(c)$. \qedhere
	\end{enumerate}
\end{proof}

\begin{lemma}\label{lemma:WS}
	Let $x=v\lambda \in {}^\mathcal{S}\widetilde{\mathcal{W}}$ with 
	$v\in\mathcal{W}$ and $\lambda\in\Lambda$. Then
	\begin{multicols}{3}
		\begin{enumerate}[(a)]
			\item $\ell(x)=\ell(\lambda)-\ell(v)$;
			\item $A_L(\lambda)\cap\mathcal{S}=A_R(v)\cap\mathcal{S}$;
			\item $C_\mathcal{W}(\lambda)=\mathcal{W}\cap x^{-1}\mathcal{W}x$.
		\end{enumerate}
	\end{multicols}
\end{lemma}

\begin{proof}
	Since $x\in {}^\mathcal{S}\widetilde{\mathcal{W}}$ we have
	$\ell(\lambda)=\ell(v^{-1}x)=\ell(x)+\ell(v^{-1})=\ell(x)+\ell(v)$ which 
	proves $(a)$.
	To prove $(b)$, let $s\in \mathcal{S}$. Then 
	$\ell(s\lambda)=\ell(sv^{-1}x)= 
	\ell(x)+\ell(sv^{-1})=\ell(\lambda)-\ell(v^{-1})+\ell(sv^{-1})$
	and so 
	$A_L(\lambda)\cap\mathcal{S}=A_L(v^{-1})\cap\mathcal{S}= 
	A_R(v)\cap\mathcal{S}$.
	Finally, we have $\mathcal{W}\cap x^{-1}\mathcal{W}x= 
	\mathcal{W}\cap \lambda^{-1} \mathcal{W}\lambda$ 
	which is $C_\mathcal{W}(\lambda)$ by Lemma \ref{lemma:costanteclassi}$(b)$.
\end{proof}

We recall that for every $J\subset \mathcal{S}$ we denote by $w_J$ the element 
of maximal length in $\mathcal{W}_J$ and we put $w_0=w_\mathcal{S}$. Then we 
have 
$\ell(w_Jw)=\ell(ww_J)=\ell(w_J)-\ell(w)$ and 
$\ell(w_Jww_J)=\ell(w)$ for every $w\in \mathcal{W}_J$.

\begin{lemma}\label{lemma:SWS}
	Let $x=v\lambda \in {}^\mathcal{S}\widetilde{\mathcal{W}}^\mathcal{S}$ with 
	$v\in\mathcal{W}$ and $\lambda\in\Lambda$ and let $Q=\mathcal{S}\cap 
	x^{-1}\mathcal{S}x$. Then
\begin{multicols}{4}
	\begin{enumerate}[(a)]
	\item $C_\mathcal{S}(\lambda)=Q$;
	\item $\mathcal{S}\subset A_R(\lambda)$;	
	\item $A_L(\lambda)\cap\mathcal{S}=Q$;
	\item $v=w_0w_Q$.
\end{enumerate}	
\end{multicols}
\end{lemma}

\begin{proof}\mbox{}
	\begin{enumerate}[$(a)$]
		\item Since $x\in {}^\mathcal{S}\widetilde{\mathcal{W}}^\mathcal{S}$, 
		by (\ref{eq:lemma2.25}) we have $\mathcal{W}\cap 
		x^{-1}\mathcal{W}x=\mathcal{W}_{Q}$ and then by Lemma 
		\ref{lemma:WS}$(c)$ we obtain  $C_\mathcal{S}(\lambda)=Q$.
		\item By $(a)$ and Lemma \ref{lemma:costanteclassi}$(c)$ we have 
		$Q=C_\mathcal{S}(\lambda)\subset A_L(\lambda)$ and so $Q\subset 
		A_R(\lambda)$. 
		Now, if $s\in \mathcal{S}\setminus Q$, then $\lambda s=v^{-1}xs$ and by 
		Lemma \ref{lemma:decwtilde} we have 
		$\ell(v^{-1}xs)=\ell(v^{-1})+\ell(x)+\ell(s)=\ell(\lambda)+1$.
		\item It follows by $(b)$ and Lemma \ref{lemma:costanteclassi}$(d)$.
		\item By $(b)$ and $(c)$ we have $\lambda\in 
		{}^Q\widetilde{\mathcal{W}}^\mathcal{S}$.
		We consider the element $\lambda w_Q$ and we prove that 
		$\mathcal{S}\subset D_L(\lambda w_Q)$. 
		If $s\in Q=C_\mathcal{S}(\lambda)$ then $\ell(s\lambda 
		w_Q)=\ell(\lambda s 
		w_Q)=\ell(\lambda)+\ell(sw_Q)<\ell(\lambda)+\ell(w_Q)=\ell(\lambda 
		w_Q)$.
		If $s\in\mathcal{S}\setminus Q$ then by $(c)$ we have 
		$\ell(s\lambda)<\ell(\lambda)$ and so by Lemma 
		\ref{lemma:compatibilita}$(a)$ we have $\mathcal{S}\subset 
		A_R(s\lambda)$
		which implies $\ell(s\lambda 
		w_Q)=\ell(s\lambda)+\ell(w_Q)<\ell(\lambda)+\ell(w_Q)=\ell(\lambda 
		w_Q)$.
		Hence $w_Q\lambda$, which is equal to $\lambda w_Q$ by 
		$(a)$, is the longest element in $ \mathcal{W}\lambda=\mathcal{W}x$ and 
		so it is equal to $w_0x$. This 
		implies $v=w_0w_Q$. \qedhere
	\end{enumerate}
\end{proof}

In summary, if $x\in{}^\mathcal{S}\widetilde{\mathcal{W}}^\mathcal{S}$ then 
$x=w_0w_Q\lambda $ where $\lambda\in\Lambda$, $Q=\mathcal{S}\cap 
x^{-1}\mathcal{S}x=C_\mathcal{S}(\lambda)=A_L(\lambda)\cap\mathcal{S}$ and 
$\mathcal{S}\subset A_R(\lambda)$. 

\begin{lemma}\label{lemma:viceversa}
	Let $\lambda\in\Lambda$ and let $Q=C_\mathcal{S}(\lambda)$. If 
	$\mathcal{S}\subset A_R(\lambda)$ then $w_0w_Q\lambda\in 
	{}^\mathcal{S}\widetilde{\mathcal{W}}^\mathcal{S}$.
\end{lemma} 

\begin{proof}
	By Lemma \ref{lemma:costanteclassi} we have $A_L(\lambda)\cap 
	\mathcal{S}=Q$.
	Proceeding as in the proof of Lemma \ref{lemma:SWS}$(d)$, we can prove that 
	$\mathcal{S}\subset D_L(\lambda w_Q)$. 
	Hence $w_Q\lambda =\lambda w_Q$ is the longest element in 
	$ \mathcal{W}\lambda$ and so $w_0w_Q\lambda \in 
	{}^\mathcal{S}\widetilde{\mathcal{W}}$.
	Now, by Lemma \ref{lemma:WS}$(a)$ we have $\ell(w_0w_Q\lambda 
	)=\ell(\lambda)-\ell(w_Q w_0)$ and, since $\lambda\in 
	\widetilde{\mathcal{W}}^\mathcal{S}$, by Lemma 
	\ref{lemma:compatibilita}$(a)$ we have $w_0w_Q\lambda \in 
	\widetilde{\mathcal{W}}^\mathcal{S}$. 
\end{proof}

\begin{prop}\label{prop:pQSS}
	We have
	$\{x\in{}^\mathcal{S}\widetilde{\mathcal{W}}^\mathcal{S}\,|\, 
	\mathcal{S}\cap x^{-1}\mathcal{S}x=Q\}=\{w_0w_Q\lambda 
	\,|\,\lambda\in\Lambda\cap \widetilde{\mathcal{W}}^\mathcal{S}, 
	C_\mathcal{S}(\lambda)=Q\}.$
\end{prop}

\begin{proof}
	One inclusion is given by Lemma \ref{lemma:SWS} while the other inclusion 
	is given by Lemmas \ref{lemma:viceversa} and \ref{lemma:SWS}$(a)$.
\end{proof}

Now, as a consequence of Lemma \ref{lemma:SWS}, we show that the rationality of 
$p_{Q,J,K}(t)$ depends on the rationality of $p_{Q,\mathcal{S},\mathcal{S}}(t)$.

\begin{corol}\label{corol:razionale2}
	Let $Q,J\subset \mathcal{S}$. Then
	\begin{equation*}
	p_{Q,J,\mathcal{S}}(t)=\sum_{Q\subset Q'\subset \mathcal{S}} 
	p^{\mathcal{S}}_{w_0w_{Q'}Qw_{Q'}w_0,J,w_0Q'w_0}(t) 
	\;p_{Q',\mathcal{S},\mathcal{S}}(t).
	\end{equation*}
\end{corol}

\begin{proof}
	We want to prove the equality
	\[\{x\in{}^J\widetilde{\mathcal{W}}^\mathcal{S}\,|\,\mathcal{S}\cap 
	x^{-1}Jx=Q\}=
	\bigsqcup_{Q\subset Q'\subset \mathcal{S}}
	\bigsqcup_{\substack{z\in 
			{}^{\mathcal{S}}\widetilde{\mathcal{W}}^{\mathcal{S}}\\ 
			\mathcal{S}\cap z^{-1}\mathcal{S}z=Q'}}
	\bigsqcup_{\substack{y\in{}^J\mathcal{W}^{w_0Q'w_0}\\ w_0Q'w_0\cap 
	y^{-1}Jy=w_0w_{Q'}Qw_{Q'}w_0}}
	\{yz\}.
	\]
	Let $x\in p_{Q,J,\mathcal{S}}$. By Lemma \ref{lemma:razionale1} we have 
	$x=yz$ with $z\in {}^{\mathcal{S}}\widetilde{\mathcal{W}}^{\mathcal{S}}$ 
	and $y\in {}^J\mathcal{W}^{\mathcal{S}\cap z\mathcal{S}z^{-1}}$.
	If we put $Q':=\mathcal{S}\cap z^{-1}\mathcal{S}z$, then, by Lemma 
	\ref{lemma:SWS} we have $z=w_0w_{Q'}\lambda$ with 
	$\lambda\in\Lambda$ and $C_\mathcal{S}(\lambda)=Q'$. 
	Hence, $\mathcal{S}\cap z\mathcal{S}z^{-1}=zQ'z^{-1}=w_0w_{Q'}\lambda 
	Q'\lambda^{-1}w_{Q'}w_0=w_0 Q' w_0$ and so $y\in\mathcal{W}^{w_0Q'w_0}$.
	Moreover, since $x\in {}^J\widetilde{\mathcal{W}}^\mathcal{S}$ and $z\in 
	{}^{\mathcal{S}}\widetilde{\mathcal{W}}^{\mathcal{S}}$ we have 
	$\mathcal{W}_{Q}=\mathcal{W}\cap x^{-1}\mathcal{W}_Jx\subset 
	\mathcal{W}\cap z^{-1}y^{-1}\mathcal{W}yz=\mathcal{W}\cap 
	z^{-1}\mathcal{W}z=\mathcal{W}_{Q'}$ and so $Q\subset Q'$. Finally, we have 
	$w_0Q'w_0\cap y^{-1}Jy=zQ'z^{-1}\cap y^{-1}Jy=z(Q'\cap 
	x^{-1}Jx)z^{-1}=zQz^{-1}$ which is equal to $w_0w_{Q'}Qw_{Q'}w_0$ since 
	$Q\subset 
	Q'=C_\mathcal{S}(\lambda)$.	
	On the other hand, fix $Q'\subset\mathcal{S}$ containing $Q$ and let 
	$H=w_0w_{Q'}Qw_{Q'}w_0$. Let 
	$z\in {}^{\mathcal{S}}\widetilde{\mathcal{W}}^{\mathcal{S}}$ such that 
	$\mathcal{S}\cap z^{-1}\mathcal{S}z=Q'$ and 
	$y\in{}^J\mathcal{W}^{w_0Q'w_0}$ such that $w_0Q'w_0\cap 
	y^{-1}Jy=H$. By Lemma \ref{lemma:compatibilita}$(b)$, the element $x:=yz$ 
	belongs to $ {}^J\widetilde{\mathcal{W}}^\mathcal{S}$.
	As before, we have $z=w_0w_{Q'}\lambda$ with 
	$\lambda\in\Lambda$ and $C_\mathcal{S}(\lambda)=Q'$. Moreover, 
	$\mathcal{W}_{\mathcal{S}\cap x^{-1}Jx}=\mathcal{W}\cap 
	x^{-1}\mathcal{W}_Jx\subset \mathcal{W}\cap 
	z^{-1}y^{-1}\mathcal{W}yz=\mathcal{W}\cap 
	z^{-1}\mathcal{W}z=\mathcal{W}_{Q'}$ and so 
	$\mathcal{S}\cap x^{-1}Jx=Q'\cap x^{-1}Jx=z^{-1}(w_0Q'w_0\cap 
	y^{-1}Jy)z=z^{-1}Hz=\lambda^{-1}Q\lambda$ which is equal to $Q$ since 
	$Q\subset Q'=C_\mathcal{S}(\lambda)$.
\end{proof}

\begin{corol}\label{corol:pQJK}
	Let $Q,J,K\subset\mathcal{S}$. Then
	\begin{equation*}
	p_{Q,J,K}(t)=\sum_{Q'\subset \mathcal{S}}
	\sum_{Q'\subset Q''\subset \mathcal{S}}
	p^{\mathcal{S}}_{Q,Q',K}(t) \;
	p^{\mathcal{S}}_{w_0w_{Q''}Q'w_{Q''}w_0,J,w_0Q''w_0}(t) \;
	p_{Q'',\mathcal{S},\mathcal{S}}(t).
\end{equation*}
\end{corol}

\begin{proof}
	Il follows by Proposition \ref{prop:razionale2} and Corollary 
	\ref{corol:razionale2}.
\end{proof}

Hence, if we prove that $p_{Q,\mathcal{S},\mathcal{S}}(t)$ is a rational 
function for every 
$Q\subset\mathcal{S}$ 
then $p_{Q,J,K}(t)$ is rational for every $Q,J,K\subset\mathcal{S}$. 
By Proposition \ref{prop:pQSS} and Lemma \ref{lemma:WS}$(a)$ we obtain
\begin{equation}\label{eq:pQSS}
p_{Q,\mathcal{S},\mathcal{S}}(t)=t^{\ell(w_Q)-\ell(w_0)}
\sum_{\substack{\lambda\in\Lambda\cap \widetilde{\mathcal{W}}^\mathcal{S}\\ 
		C_\mathcal{S}(\lambda)=Q}}t^{\ell(\lambda)}
	=t^{\ell(w_Q)-\ell(w_0)}
	\sum_{\substack{\lambda\in\Lambda\cap 
	{}^\mathcal{S}\widetilde{\mathcal{W}}\\ 
			C_\mathcal{S}(\lambda)=Q}}t^{\ell(\lambda)}.
\end{equation}
Now, by 4 of Lemma \ref{lemma:length} we have $\lambda=t(v)\in\Lambda\cap 
{}^\mathcal{S}\widetilde{\mathcal{W}}$ if and only if 
$\langle\alpha,v\rangle\geq 0$ for every $\alpha\in\Sigma$. Moreover we have 
$s_{\alpha,0}\in 
C_\mathcal{S}(t(v))$ if and only if $\langle\alpha,v\rangle=0$. Hence, by 4 of
Lemma \ref{lemma:length} we obtain
\begin{equation}\label{eq:fQ}
	\sum_{\substack{\lambda\in\Lambda\cap 
	\widetilde{{}^\mathcal{S}\mathcal{W}}\\ 
	C_\mathcal{S}(\lambda)=Q}}t^{\ell(\lambda)}=
	\sum_{\substack{v\in L^{\vee}\\ 
	\langle\alpha,v\rangle> 0 \;\forall s_{\alpha,0}\in\mathcal{S}\setminus Q\\	
	\langle\alpha,v\rangle=0 \;\forall s_{\alpha,0}\in Q}}
	t^{\langle2\rho,v\rangle}
\end{equation} 	
where $2\rho=\sum_{\alpha\in\bm\Phi^+}\alpha$. 

\medskip
In the last part of this paragraph, using results of section 
\ref{section:coni}, we prove that, for every $Q\subset\mathcal{S}$, the series 
in (\ref{eq:fQ}), that we denote 
$f_Q(t)$, is a rational function.

\medskip
Let $n=|\Sigma|$. We fix an ordering 
$\Sigma=\{\alpha_1,\dots,\alpha_n\}$ of simple roots.
Then we can identify $V$ and $V^\vee$ with $\R^n$.
We recall that $L^\vee$ is the $\Z$-span of $\Sigma^\vee$ and so, if we denote 
$\bm\alpha^\vee=(\alpha_1^\vee,\dots,\alpha_n^\vee)$, every element of $L^\vee$ 
is of the form $\bm\alpha^\vee\mathbf{m}=\sum_{i=1}^n m_i\alpha_i^\vee$ 
where $\mathbf{m}=(m_1,\dots,m_n)^T\in\Z^n$.

\medskip
Let $\mathbf{C}=(\langle\alpha_i,\alpha_j^\vee\rangle)_{1\leq i,j\leq n}\in 
GL_n(\Q)$ be the Cartan matrix associated to $\bm\Phi$ and let
$\{\mathbf{e}_1,\dots,\mathbf{e}_n\}$ be the standard basis of $\R^n$. 
Since $\mathbf{C}$ is an invertible matrix with integer 
coefficients, the vector 
$\mathbf{v}_i:=|\det(\mathbf{C})|\mathbf{C}^{-1}\mathbf{e}_i$ 
belongs to $\Z^n$ for every $i\in\{1,\dots,n\}$. 
We remark that for every $v=\bm\alpha^\vee\mathbf{m}\in L^\vee$ we have 
$(\langle\alpha_1,v\rangle,\dots,\langle\alpha_n,v\rangle)^T= 
\mathbf{C}\mathbf{m}$. 
	
\begin{lemma}\label{lemma:fQ}
	Let $Q\subset\mathcal{S}$ and 
	$\mathcal{I}(Q)= 
	\{i\in\{1,\dots,n\}\,|\,s_{\alpha_i,0}\in \mathcal{S}\setminus Q\}$.
	We have
	\[\left\{v\in L^\vee\,\Big|
	\begin{array}{l}
	\langle\alpha,v\rangle> 0\; \forall s_{\alpha,0}\in \mathcal{S}\setminus 
	Q \\
	\langle\alpha,v\rangle= 0\; \forall 
	s_{\alpha,0}\in  Q
	\end{array}
	\hspace{-0,15cm}
	\right\}= 
	\bigg\{\sum_{i\in\mathcal{I}(Q)}\lambda_i\bm\alpha^\vee\mathbf{v}_i\, 
	\Big|\,
	\sum_{i\in\mathcal{I}(Q)}\lambda_i\mathbf{v}_i\in \Z^n ,
	\lambda_i\in\R_{>0}\,\forall i\in\mathcal{I}(Q)
	\bigg\}.\]
\end{lemma}
 	
\begin{proof}
Let $v\in L^\vee$ such that $\langle\alpha,v\rangle> 0$ for every 
$s_{\alpha,0}\in \mathcal{S}\setminus Q$ and $\langle\alpha,v\rangle= 0$ for 
every $s_{\alpha,0}\in  Q$.
Then we have $v=\bm\alpha^\vee\mathbf{m}$ with 
$\mathbf{m}\in\Z^n$ and 
$\mathbf{C}\mathbf{m}= 
\sum_{i\in\mathcal{I}(Q)}\langle\alpha_i,v\rangle\mathbf{e}_i$.
We obtain 
\[\mathbf{m}= 
\sum_{i\in\mathcal{I}(Q)}\langle\alpha_i,v\rangle\mathbf{C}^{-1} 
\mathbf{e}_i=
\sum_{i\in\mathcal{I}(Q)}\frac{\langle\alpha_i,v\rangle}{|\det(\mathbf{C})|} 
\mathbf{v}_i.\]
On the other hand, let 
$\mathbf{m}=\sum_{i\in\mathcal{I}(Q)}\lambda_i\mathbf{v}_i\in\Z^n$ with 
$\lambda_i>0$ for every $i\in\mathcal{I}(Q)$.
Then $\bm\alpha^\vee\mathbf{m}\in L^{\vee}$ and 
\[\mathbf{C}\mathbf{m}=
\sum_{i\in\mathcal{I}(Q)}\lambda_i\mathbf{C}\mathbf{v}_i=
\sum_{i\in\mathcal{I}(Q)}\lambda_i|\det(\mathbf{C})|\mathbf{e}_i\]
and so 
$\langle\alpha_i,\bm\alpha^\vee\mathbf{m}\rangle= 
\lambda_i|\det(\mathbf{C})|>0$
for every $i\in\mathcal{I}(Q)$ and 
$\langle\alpha_i,\bm\alpha^\vee\mathbf{m}\rangle=0$
otherwise.
\end{proof}
 	
\begin{rmk}\label{rmk:rCm}
Let $\mathbf{r}=(r_1,\dots,r_n)\in\N^n$ such that $2\rho=\sum_{i=1}^n 
r_i\alpha_i$.
Since for every $i\in\{1,\dots,n\}$ we have 
$s_{\alpha_i}(\rho)=\rho-\alpha_i$, we 
obtain $\mathbf{r}\mathbf{C}= 
(\langle2\rho,\alpha_1^\vee\rangle,\dots,\langle2\rho,\alpha_n^\vee\rangle) 
=(2,\dots,2)=\mathbf{2}$.
Then, for every $v=\bm\alpha^\vee\mathbf{m}\in L^\vee$ we have 
$\langle 2\rho,v\rangle=\mathbf{r}\mathbf{C}\mathbf{m}=\mathbf{2}\mathbf{m}$.
\end{rmk}

For every $Q\subset \mathcal{S}$, let 
\begin{equation}\label{eq:conoQ}
\mathcal{C}(Q)=\bigg\{\sum_{i\in\mathcal{I}(Q)}\lambda_i\mathbf{v}_i 
\,\Big|\, \lambda_i\in\R_{\geq 0} \,\forall\, i\in\mathcal{I}(Q)\bigg\}
\end{equation}
be the rational simplicial $|\mathcal{I}(Q)|$-cone in $\R^n$ associated to $Q$. 
 	
\begin{lemma}\label{lemma:fQ2}
	Let $Q\subset\mathcal{S}$, $\mathcal{C}=\mathcal{C}(Q)$ and 
	$\mathbf{t}=(t,\dots,t)$. Then 
	$f_Q(t)=\sigma_{\mathcal{C}^o}(\mathbf{t}^{\mathbf{2}})$.
\end{lemma}
 	
\begin{proof}
 Using Lemma \ref{lemma:fQ} and Remark \ref{rmk:rCm}, we obtain 
 \[\sum_{\substack{v\in L^{\vee}\\ 
	\langle\alpha,v\rangle> 0 \;\forall s_\alpha\in\mathcal{S}\setminus Q\\	
	\langle\alpha,v\rangle=0 \;\forall s_\alpha\in Q}}
	t^{\langle2\rho,v\rangle}=
	\sum_{\mathbf{m}\in\mathcal{C}^o\cap\Z^n} 
	t^{\langle2\rho,\bm\alpha^\vee\mathbf{m}\rangle}=
	\sum_{\mathbf{m}\in\mathcal{C}^o\cap\Z^n}t^{\mathbf{2}\mathbf{m}}
	=	\sum_{\mathbf{m}\in\mathcal{C}^o\cap\Z^n}
	(\mathbf{t}^{\mathbf{2}})^\mathbf{m}=
	\sigma_{\mathcal{C}^o}(\mathbf{t}^{\mathbf{2}}).\qedhere
 \]
\end{proof}
 	
\begin{teor}\label{teor:riassunto}
	Let $\bm\Phi$ be an irreducible reduced finite root system in a finite 
	dimensional real vector space, let $\widetilde{\mathcal{W}}$ be the 
	associated affine Weyl group and let $\mathcal{S}$ as in 
	paragraph \ref{par:affineWeyl}. 
	\begin{enumerate}[(a)]
		\item $p_{Q,J,K}(t)$ defined in (\ref{eq:defpQJK}) is a rational 
		function for every $Q,J,K\subset \mathcal{S}$.
		\item ${}^J \widetilde{\mathcal{W}}^K(t)$ is a rational function  
		for every $J,K\subset \mathcal{S}$.
		\item $N_{\widetilde{\mathcal{W}}}(\mathcal{W}_J)(t)$ is a 
		rational function for 
		every $J\subset \mathcal{S}$.
	\end{enumerate}
\end{teor}
 	
\begin{proof}
Point $(a)$ comes from Corollary \ref{corol:pQJK}, formula (\ref{eq:pQSS}), 
Lemma \ref{lemma:fQ2} and Proposition \ref{prop:cono}. 
Consequently, $(b)$ follows by (\ref{eq:JWK}) while $(c)$ by paragraph 
\ref{sec:normalizer}.
\end{proof}

\section{Examples}\label{sec:examples}
We resume some results of previous sections in order to calculate 
explicitly $p_{Q,J,K}(t)$ and so ${}^J 
\widetilde{\mathcal{W}}^K(t)$ for every $Q,J,K\subset\mathcal{S}$.

\medskip
Let $\Sigma=\{\alpha_1,\dots,\alpha_n\}$, $\mathbf{C}$ be the Cartan 
matrix and
$\mathbf{v}_i:=|\det(\mathbf{C})|\mathbf{C}^{-1}\mathbf{e}_i$.
Thanks to Remark \ref{rmk:vettoriridotti}, in (\ref{eq:conoQ}) we can replace 
each  $\mathbf{v}_i=(v_{i,1},\dots,v_{i,n})^T$
by $\mathbf{w}_i:=\mathbf{v}_i/\gcd(v_{i,1},\dots,v_{i,n})$.
For every $Q\subset \mathcal{S}$, by formula (\ref{eq:pQSS})  we 
have $p_{Q,\mathcal{S},\mathcal{S}}(t)=t^{\ell(w_{Q})-\ell(w_0)} f_Q(t)$ and by 
Lemma \ref{lemma:fQ2} and Proposition \ref{prop:cono} we have
\[f_Q(t)= \sum_{R\supseteq Q} (-1)^{|R|-|Q|}
\frac{\sum_{\mathbf{m}\in \Pi(\mathcal{C}(R))\cap\Z^n} 
	t^{\mathbf{2}\mathbf{m}}}
{\prod_{i\in\mathcal{I}(R)}(1-t^{\mathbf{2}\mathbf{w}_i})} \]
where
\[\Pi(\mathcal{C}(R))=\bigg\{\sum_{i\in\mathcal{I}(R)}\lambda_i\mathbf{w}_i 
\,\Big|\, 
0\leq\lambda_i<1 \,\forall\, i\in\mathcal{I}(R)\bigg\}
\quad\text{and}\quad 
\mathcal{I}(R)=\{i\in\{1,\dots,n\}\,|\,s_{\alpha_i}\in 
\mathcal{S}\setminus R\}.
\]

\begin{rmk}\label{rmk:calcolofQ}
	If $\Pi(\mathcal{C}(R))\cap\Z^n=\{(0,\dots,0)^T\}$ for 
	every $Q\subset R\subset  \mathcal{S}$, we obtain
	\begin{align*}
	f_Q(t)&
	=\frac{1}
	{\prod_{i\in\mathcal{I}(Q)}(1-t^{\mathbf{2}\mathbf{w}_i})}
	\sum_{R\supseteq Q} 
	(-1)^{|R|-|Q|}\prod_{i\in\mathcal{I}(Q)\setminus\mathcal{I}(R)} 
	(1-t^{\mathbf{2}\mathbf{w}_i})\\
	&=\frac{1}
	{\prod_{i\in\mathcal{I}(Q)}(1-t^{\mathbf{2}\mathbf{w}_i})} 
	\bigg(\prod_{i\in\mathcal{I}(Q)} 
	(X-(1-t^{\mathbf{2}\mathbf{w}_i}))\bigg)_{|_{X=1}}\\
	&=\frac{t^{\sum_{i\in\mathcal{I}(Q)}\mathbf{2}\mathbf{w}_i}}
	{\prod_{i\in\mathcal{I}(Q)}(1-t^{\mathbf{2}\mathbf{w}_i})} 
	.\end{align*}
\end{rmk}
After that, using Corollary \ref{corol:razionale2} we can calculate the 
$2^n\times 2^n$ matrix 
$M_{\mathcal{S}}:=\big(p_{Q,J,\mathcal{S}}(t)\big)_{Q\subset 
	\mathcal{S},J\subset \mathcal{S}}$.
Finally, for every $K\subset\mathcal{S}$ we can calculate the polynomial 
matrix 
\[M_{K,\mathcal{S}}:=\Big(p^{\mathcal{S}}_{Q,J,K}(t)\Big)_{Q\subset K,J\subset 
\mathcal{S}}\in \mathrm{Mat}_{2^{|K|}\times 2^{n}}(\Z[t]).\]
Hence, by Proposition \ref{prop:razionale2}, we have
$\big(p_{Q,J,K}(t)\big)_{Q\subset K,J\subset 
	\mathcal{S}}=M_{K,\mathcal{S}}M_{\mathcal{S}}$ and 
 ${}^J 
\widetilde{\mathcal{W}}^K(t)$ is the sum of the $J$-th column of this matrix.

\medskip
In the next paragraphs we will give some examples when the root system  
$\bm\Phi$ is irreducible. 

\subsection{Root system of type $A_{n+1}$}
We have $\mathcal{W}\cong\mathfrak{S}_{n+1}$.
Let $\alpha_i=\mathbf{e}_i-\mathbf{e}_{i+1}$ for every $i\in\{1,\dots,n\}$.
The Cartan matrix is
\[\mathbf{C}=
\begin{pmatrix} 
	2 		& -1 		&  0 		&  \cdots	& 0 		\\
   -1 		&  \ddots	& \ddots	& \ddots	& \vdots  	\\
	0  		& \ddots 	& \ddots 	& \ddots 	& 0	    	\\
	\vdots 	& \ddots 	& \ddots 	& \ddots 	& -1	    \\
	0  		& \cdots 	& 0 		& -1 		& 2	    	\\
\end{pmatrix}
\in M_n(\Z)\]
which has determinant $n+1$ and so we have $\mathbf{v}_i=\sum_{k=1}^n 
\min\{i,k\}(n+1-\max\{i,k\})\mathbf{e}_k$ for every $i\in\{1,\dots n\}$.

\subsubsection{$A_3$}
We have 
$\mathbf{C}=\left(\begin{smallmatrix}2&-1\\-1&2\end{smallmatrix}\right)$, 
$\mathbf{v}_1=\mathbf{w}_1=\left(\begin{smallmatrix}2\\1\end{smallmatrix}\right)$
 and 
$\mathbf{v}_2=\mathbf{w}_2=\left(\begin{smallmatrix}1\\2\end{smallmatrix}\right)$.
Let $Q_1=\{\alpha_1\}$ and $Q_2=\{\alpha_2\}$. 
Then we have 
$\Pi(\mathcal{C}(Q_1))\cap\Z^2=\Pi(\mathcal{C}(Q_2))\cap\Z^2=
\{\left(\begin{smallmatrix}0\\0\end{smallmatrix}\right)\}$ 
and $\Pi(\mathcal{C}(\emptyset))\cap\Z^2=
\{\left(\begin{smallmatrix}0\\0\end{smallmatrix}\right),
\left(\begin{smallmatrix}1\\1\end{smallmatrix}\right),
\left(\begin{smallmatrix}2\\2\end{smallmatrix}\right)
\}$.
We obtain
\[f_{Q_1}(t)=f_{Q_2}(t)=\frac{1}{1-t^6}-1=
\frac{t^6}{1-t^6}
\quad\text{and}\quad f_{\emptyset}(t)
=\frac{1+t^4+t^8}{(1-t^6)^2}-\frac{2}{1-t^6}+1=
\frac{t^4(1-t^2+t^4)}{(1-t^2)(1-t^6)}
\]
and
\[M_{\mathcal{S}}=\frac{1}{(1-t^2)(1-t^6)}
\begin{pmatrix} 
(1+t+t^2)(1+t^3)			& t(1+t+t^4) 	&   
t(1+t+t^4)&  t(1-t^2+t^4)\\
0 	& 1-t^2	& t^4(1-t^2)	&  t^4(1-t^2)	\\
0 	& t^4(1-t^2) & 1-t^2  	& t^4(1-t^2)    	\\
0 	& 0 				& 0 				& 
(1-t^2)(1-t^6)	    				\\
\end{pmatrix}.
\]
Moreover, we have $M_{\emptyset,\mathcal{S}}=
\begin{pmatrix} 
(1+t)(1+t+t^2) & 1+t+t^2	 & 1+t+t^2		& 1
\end{pmatrix}
$ and 
\[ M_{Q_1,\mathcal{S}}=
\begin{pmatrix} 
1+t+t^2		& t	&   1 	&  0\\
0 			& 1	& t^2	& 1	\\
\end{pmatrix},
\quad M_{Q_2,\mathcal{S}}=
\begin{pmatrix} 
1+t+t^2		& 1	&   t 	&  0\\
0 			& t^2	& 1	& 1	\\
\end{pmatrix}.
\]

\subsubsection{$A_4$}
We have 
$\mathbf{C}= 
\left(\begin{smallmatrix}2&-1&0\\-1&2&-1\\0&-1&2\end{smallmatrix}\right)$,
$\mathbf{w}_1= 
\left(\begin{smallmatrix}3\\2\\1\end{smallmatrix}\right)$, 
$\mathbf{w}_2= \left(\begin{smallmatrix}1\\2\\1\end{smallmatrix}\right)$ and 
$\mathbf{w}_3=\left(\begin{smallmatrix}1\\2\\3\end{smallmatrix}\right)$.
Let $Q_1=\{\alpha_1\}$, $Q_2=\{\alpha_2\}$, $Q_3=\{\alpha_3\}$,  
$Q_{12}=\{\alpha_1,\alpha_2\}$, $Q_{13}=\{\alpha_1,\alpha_3\}$ and 
$Q_{23}=\{\alpha_2,\alpha_3\}$. 
Then we have 
\[
\Pi(\mathcal{C}(Q))\cap\Z^3=
\left\{
\begin{array}{ll}
\{(0,0,0)^T\} 				& \text{if }Q\in\{Q_{12},Q_{13},Q_{23}\}\\
\{(0, 0, 0)^T,  (1, 2, 2)^T\}	& \text{if }Q=Q_{1}\\
\{(0, 0, 0)^T, (1, 1, 1)^T, (2, 2, 2)^T, (3, 3, 3)^T\}	& \text{if }Q=Q_{2}\\
\{(0, 0, 0)^T, (2, 2, 1)^T\} 	& \text{if }Q=Q_{3}\\
\Bigg\{
\begin{array}{l}
(0,0,0)^T, (1,1,1)^T, (1,2,2)^T, (2,2,1)^T, \\ 
(2,2,2)^T, (2,3,3)^T, (3,3,2)^T, (3,3,3)^T
\end{array}
\Bigg\} & \text{if }Q=\emptyset.
\end{array}
\right.
\]
We obtain
\[
f_{Q}(t)=
\left\{
\begin{array}{cl}
\displaystyle{\frac{t^{12}}{1-t^{12}}}		& \text{if }Q\in\{Q_{12}, 
Q_{23}\}\vspace{0,2cm}\\
\displaystyle{\frac{t^{8}}{1-t^{8}}}		& \text{if }Q=Q_{13}\vspace{0,2cm}\\
\displaystyle{\frac{t^{10} (1+t^{10})}{(1-t^{12})(1-t^8)}} & \text{if 
}Q\in\{Q_{1},Q_{3}\}\vspace{0,2cm}\\
\displaystyle{\frac{t^6(1+t^6+t^{12}+t^{18})}{(1-t^{12})^2}}	& 
\text{if }Q=Q_{2}\vspace{0,2cm}\\
\displaystyle{\frac{t^{14} (1+2t^2+t^{12})}{(1-t^{12})(1-t^8)(1-t^6)}}	& 
\text{if }Q=\emptyset.
\end{array}
\right.
\]

\subsection{Root system of type $B_{n}$}
We have $\mathcal{W}\cong\mathfrak{S}_{n}\ltimes (\Z/2\Z)^n$.
Let $\alpha_i=\mathbf{e}_i-\mathbf{e}_{i+1}$ for every $i\in\{1,\dots,n-1\}$ 
and $\alpha_n=\mathbf{e}_n$.
The Cartan matrix is
\[\mathbf{C}=
\begin{pmatrix} 
2 		& -1 		&  0 		&  \cdots	& 0 	 & 0		\\
-1 		&  \ddots	& \ddots	& \ddots	& \vdots & \vdots  	\\
0  		& \ddots 	& \ddots 	& \ddots 	& 0	     & \vdots	\\
\vdots 	& \ddots 	& \ddots 	& 2		 	& -1	 & 0   		\\
0  		& \cdots 	& 0 		& -1 		& 2	     &-2		\\
0  		& \cdots 	& \cdots	& 0 		& -1	 &2		\\
\end{pmatrix}
\in M_n(\Z)\]
which has determinant $2$.
Hence, for every $i\in\{i,\dots,n\}$, we have 
\[
\mathbf{w}_i=
\left\{
\begin{array}{ll}
\displaystyle{\sum_{k=1}^{i-1}2k\mathbf{e}_k+\sum_{k=i}^{n-1}2i\mathbf{e}_k 
+i\mathbf{e}_n=(2,4,\dots,2(i-1),2i,\dots,2i,i)^T}		& \text{if } i \text{ 
odd}\vspace{0,2cm}\\
\displaystyle{\sum_{k=1}^{i-1}k\mathbf{e}_k+\sum_{k=i}^{n-1}i\mathbf{e}_k 
+\frac{i}{2}\mathbf{e}_n=(1,2,\dots,i-1,i,\dots,i,i/2)^T}		
& \text{if } i \text{ even}.
\end{array}
\right.
\]
which implies $\mathbf{2} 
\mathbf{w}_i=i(2n-i)$ if $i$ odd and  
$\mathbf{2} \mathbf{w}_i=i(2n-i)/2$ if $i$ even.

\subsubsection{$B_3$}
We have 
$\mathbf{C}= 
\left(\begin{smallmatrix}2&-1&0\\-1&2&-2\\0&-1&2\end{smallmatrix}\right)$,
$\mathbf{w}_1= 
\left(\begin{smallmatrix}2\\2\\1\end{smallmatrix}\right)$, 
$\mathbf{w}_2= \left(\begin{smallmatrix}1\\2\\1\end{smallmatrix}\right)$ and 
$\mathbf{w}_3=\left(\begin{smallmatrix}2\\4\\3\end{smallmatrix}\right)$.
Let $Q_1=\{\alpha_1\}$, $Q_2=\{\alpha_2\}$, $Q_3=\{\alpha_3\}$,  
$Q_{12}=\{\alpha_1,\alpha_2\}$, $Q_{13}=\{\alpha_1,\alpha_3\}$ and 
$Q_{23}=\{\alpha_2,\alpha_3\}$. 
Then we have 
\[
\Pi(\mathcal{C}(Q))\cap\Z^3=
\left\{
\begin{array}{ll}
\{(0,0,0)^T\} 				& \text{if }Q\in\{Q_1,Q_3,Q_{12},Q_{13},Q_{23}\}\\
\{(0, 0, 0)^T,  (2, 3, 2)^T\}	& \text{if }Q\in\{\emptyset, Q_2\}
\end{array}
\right.
\]
Hence, we have
\[f_{\emptyset}(t)=\frac{t^{22}(1+t^{14})}{(1-t^8)(1-t^{10})(1-t^{18})}
\quad \text{ and 
}\quad
f_{Q_{2}}(t)=\frac{t^{14}(1+t^{14})}{(1-t^{10})(1-t^{18})}
\]
and, using Remark \ref{rmk:calcolofQ}, we can easily compute $f_Q(t)$ for 
$Q\in\{Q_1,Q_3,Q_{12},Q_{13},Q_{23}\}$.

\subsection{Root system of type $C_{n}$}
We have $\mathcal{W}\cong\mathfrak{S}_{n}\ltimes (\Z/2\Z)^n$.
Let $\alpha_i=\mathbf{e}_i-\mathbf{e}_{i+1}$ for every $i\in\{1,\dots,n-1\}$ 
and $\alpha_n=2\mathbf{e}_n$.
The Cartan matrix is
\[\mathbf{C}=
\begin{pmatrix} 
2 		& -1 		&  0 		&  \cdots	& 0 	 & 0		\\
-1 		&  \ddots	& \ddots	& \ddots	& \vdots & \vdots  	\\
0  		& \ddots 	& \ddots 	& \ddots 	& 0	     & \vdots	\\
\vdots 	& \ddots 	& \ddots 	& 2		 	& -1	 & 0   		\\
0  		& \cdots 	& 0 		& -1 		& 2	     &-1		\\
0  		& \cdots 	& \cdots	& 0 		& -2	 &2		\\
\end{pmatrix}
\in M_n(\Z)\]
which has determinant $2$.
Hence, for every $i\in\{i,\dots,n\}$, we have 
$\mathbf{w}_i=\sum_{k=1}^{i-1} 
k\mathbf{e}_k+\sum_{k=i}^{n} 
i\mathbf{e}_k=(1,2,\dots,i,\dots,i)^T$
which implies $\mathbf{2} 
\mathbf{w}_i=i(2n+1-i)$ and $\Pi(\mathcal{C}(Q))\cap\Z^n=\{(0,\dots,0)^T\}$ for 
every $Q\subset \mathcal{S}$. 
Hence, using Remark \ref{rmk:calcolofQ}, we can easily compute $f_Q(t)$ for 
every $Q\subset \mathcal{S}$.

\subsection{Root system of type $F_{4}$}
We have $\mathcal{W}\cong\mathfrak{S}_{3}\ltimes(\mathfrak{S}_{4}\ltimes 
(\Z/2\Z)^2)$.
Let $\alpha_1=\mathbf{e}_2-\mathbf{e}_{3}$, 
$\alpha_1=\mathbf{e}_3-\mathbf{e}_{4}$, $\alpha_3=\mathbf{e}_4$ and 
$\alpha_4=\frac{1}{2}(\mathbf{e}_1-\mathbf{e}_{2}-\mathbf{e}_{3}-\mathbf{e}_{4})$.
The Cartan matrix is
\[\mathbf{C}=
\begin{pmatrix} 
2 		& -1 		&  0 	 & 0		\\
-1 		&  2	 	& -2	 & 0   		\\
0  		& -1 		&  2	 & -1		\\
0  		& 0 		& -1	 & 2		\\
\end{pmatrix}
\]
which has determinant $1$.
Hence, we have 
$\mathbf{w}_1=(2,3,2,1)^T$,
$\mathbf{w}_2=(3,6,4,2)^T$,
$\mathbf{w}_3=(4,8,6,3)^T$ and
$\mathbf{w}_4=(2,4,3,2)^T$
which implies $\Pi(\mathcal{C}(Q))\cap\Z^4=\{(0,0,0,0)^T\}$ for 
every $Q\subset \mathcal{S}$. 
Hence, using Remark \ref{rmk:calcolofQ}, we can easily calculate $f_Q(t)$ for 
every $Q\subset \mathcal{S}$.

\subsection{Root system of type $G_{2}$}
We have $\mathcal{W}\cong D_6$, the dihedral group of order $12$.
Let $\alpha_1=-2\mathbf{e}_1+\mathbf{e}_2+\mathbf{e}_3$ and 
$\alpha_2=\mathbf{e}_1-\mathbf{e}_2$.
The Cartan matrix is
$\mathbf{C}=\left(\begin{smallmatrix}2&-3\\-1&2\end{smallmatrix}\right)$
which has determinant $1$. 
Hence, we have 
$\mathbf{w}_1=\left(\begin{smallmatrix}2\\1\end{smallmatrix}\right)$
and 
$\mathbf{w}_2=\left(\begin{smallmatrix}3\\2\end{smallmatrix}\right)$which 
implies $\Pi(\mathcal{C}(Q))\cap\Z^2= 
\{\left(\begin{smallmatrix}0\\0\end{smallmatrix}\right)\}$ for 
every $Q\subset \mathcal{S}$. 
Hence, using Remark \ref{rmk:calcolofQ}, we can easily compute $f_Q(t)$ for 
every $Q\subset \mathcal{S}$.

\bibliographystyle{plain}
\bibliography{Poincare}

\begin{thebibliography}{10}

\bibitem{AB}
P.~Abramenko and K.~S. Brown.
\newblock {\em Buildings}, volume 248 of {\em Graduate Texts in Mathematics}.
\newblock Springer, New York, 2008.
\newblock Theory and applications.

\bibitem{BR}
M.~Beck and S.~Robins.
\newblock {\em Computing the continuous discretely}.
\newblock Undergraduate Texts in Mathematics. Springer, New York, 2007.
\newblock Integer-point enumeration in polyhedra.

\bibitem{B98}
R.~E. Borcherds.
\newblock Coxeter groups, {L}orentzian lattices, and {$K3$} surfaces.
\newblock {\em International Mathematics Research Notices}, (19):1011--1031,
  1998.

\bibitem{BH}
B.~Brink and R.~B. Howlett.
\newblock Normalizers of parabolic subgroups in {C}oxeter groups.
\newblock {\em Invent. Math.}, 136(2):323--351, 1999.

\bibitem{Brown82}
K.~S. Brown.
\newblock {\em Cohomology of groups}, volume~87 of {\em Graduate Texts in
  Mathematics}.
\newblock Springer-Verlag, New York, 1982.

\bibitem{Brown00}
K.~S. Brown.
\newblock The coset poset and probabilistic zeta function of a finite group.
\newblock {\em Journal of Algebra}, 225(2):989--1012, 2000.

\bibitem{CCW}
I.~Castellano, G.~Chinello, and T.~Weigel.
\newblock The {H}attori-{S}tallings rank, the {E}uler characteristic and the
  {$\zeta$}-functions of totally disconnected locally compact group.
\newblock preprint, 2020.

\bibitem{D82}
V.~V. Deodhar.
\newblock On the root system of a {C}oxeter group.
\newblock {\em Communications in Algebra}, 10(6):611--630, 1982.

\bibitem{Harder}
G.~Harder.
\newblock A {G}auss-{B}onnet formula for discrete arithmetically defined
  groups.
\newblock {\em Ann. Sci. {\'E}cole Norm. Sup. (4)}, 4:409--455, 1971.

\bibitem{Hum}
J.~E. Humphreys.
\newblock {\em Reflection groups and {C}oxeter groups}, volume~29 of {\em
  Cambridge Studies in Advanced Mathematics}.
\newblock Cambridge University Press, Cambridge, 1990.

\bibitem{Mac}
I.~G. Macdonald.
\newblock {\em Affine {H}ecke algebras and orthogonal polynomials}.
\newblock Cambridge University Press, 2003.

\bibitem{Nuida}
K.~Nuida.
\newblock On centralizers of parabolic subgroups in {C}oxeter groups.
\newblock {\em Journal of Group Theory}, 14(6):891--930, 2011.

\bibitem{Serre}
J.-P. Serre.
\newblock Cohomologie des groupes discrets.
\newblock In {\em Prospects in mathematics ({P}roc. {S}ympos., {P}rinceton
  {U}niv., {P}rinceton, {N}.{J}., 1970)}, volume~70, pages 77--169. Princeton
  University Press, 1971.

\bibitem{Vis}
S.~Viswanath.
\newblock Poincar\'{e} series of subsets of affine {W}eyl groups.
\newblock {\em Proceedings of the American Mathematical Society},
  136(11):3735--3740, 2008.

\end{thebibliography}

\end{document}